\documentclass[onefignum,onetabnum]{siamart250211}

\usepackage{graphicx} 
\usepackage{mathtools,amssymb}
\usepackage{xcolor}
\usepackage{enumitem}

\newtheorem{remark}{Remark}
\newtheorem{example}{Example}

\usepackage[mode=buildnew]{standalone}
\usepackage{hyperref}
\usepackage[capitalize]{cleveref}
\usepackage{diagbox}
\usepackage{multirow}
\usepackage{booktabs}

\definecolor{darkerblue}{rgb}{0.1,0,0.8}

\newcommand{\grad}{\nabla}

\newcommand{\uu}{\mathbf{u}}
\newcommand{\vv}{\mathbf{v}}
\newcommand{\ww}{\mathbf{w}}
\newcommand{\phiphi}{\boldsymbol{\phi}}
\newcommand{\nn}{\mathbf{n}}
\newcommand{\ff}{\mathbf{f}}
\newcommand{\ffone}{\mathbf{f}}
\newcommand{\fftwo}{\uu_D}
\newcommand{\ffthree}{\mathbf{t}}
\newcommand{\HH}{\mathbf{H}}
\newcommand{\LL}{\mathbf{L}}
\newcommand{\VV}{\mathbf{V}}

\newcommand{\coarsevol}{\nu}

\newcommand{\Hc}{H} 
\newcommand{\DW}{W} 
\newcommand{\DL}{L} 

\newcommand{\jump}[1]{\ensuremath{[#1]} }

\title{On a robust inf-sup condition for the Stokes problem in slender domains -- with application to preconditioning}

\author{Espen Sande\footnotemark[2] \footnotemark[1] \and Timo Koch\footnotemark[4] \footnotemark[2] \footnotemark[3]\and
  Miroslav Kuchta\footnotemark[2]\and
  Kent-Andre Mardal\footnotemark[2] \footnotemark[3]
}

\begin{document}

\maketitle

\renewcommand{\thefootnote}{\fnsymbol{footnote}}
\footnotetext[1]{Corresponding author}
\footnotetext[2]{Department of Numerical Analysis and Scientific Computing, Simula Research Laboratory, Oslo, Norway.
    \email{sande@simula.no}, \email{miroslav@simula.no}, \email{kent-and@simula.no}}
\footnotetext[3]{Department of Mathematics, University of Oslo, Oslo, Norway.}
\footnotetext[4]{Department of Hydromechanics and Modeling of Hydrosystems, University of Stuttgart, Stuttgart, Germany. \email{timo.koch@iws.uni-stuttgart.de}}

\renewcommand{\thefootnote}{\arabic{footnote}}

\begin{abstract}
  We identify a norm on the pressure variable in the Stokes equation that allows us to prove a continuous inf-sup condition with a constant independent of the domain's aspect ratio. This is in contrast to the standard inf-sup constant, which breaks down as the aspect ratio increases. We further apply our result to construct robust operator preconditioners for the Stokes problem in slender domains. Several numerical examples illustrate the theory.
\end{abstract}




\section{Introduction}\label{sec:intro}
Let us consider the following Stokes problem: 
\begin{subequations}
\begin{align}
\label{eq:stokes1}
-\Delta \uu - \nabla p &= \ffone, &&\mbox{ in } \Omega\subset\mathbb{R}^d, \\
\label{eq:stokes2}
\nabla \cdot \uu &= 0, &&\mbox{ in } \Omega, \\
\label{eq:stokes3}
 \uu &= \fftwo, &&\mbox{ on } \partial \Omega_D, \\ 
            \label{eq:stokes4}
            \nabla \uu \cdot \nn + p \nn &= \ffthree, &&\mbox{ on } \partial \Omega_N . 
\end{align}
\label{eq:stokes}
\end{subequations}
Here $\uu$ and $p$ are the unknown fluid velocity and a rescaled\footnote{We assume a constant fluid density and viscosity field, and the viscosity is absorbed in the rescaled pressure unknown. Boundary data and right-hand side are scaled accordingly. The pressure sign is flipped to obtain a symmetric Stokes operator.} pressure unknown, 
$\ffone$ represents rescaled body forces,  
$\fftwo$ the velocity data on the boundary, and $\ffthree$ represents a rescaled traction vector
on the boundary. The above equation is chosen for simplicity of notation, and, under the boundary conditions we are considering, our analysis also covers the more physically relevant case of $-\Delta \uu$ in \cref{eq:stokes1} being replaced by $-\nabla\cdot(\mu\nabla^S\uu)$, where $\mu$ is the viscosity and $\nabla^S$ is the symmetric gradient. We address this further below.
In the following, we consider the case where $\Omega$ is a slender domain characterized by a length scale $\DL$ and a width scale $\DW$. We will focus on large-aspect-ratio domains, $\DW \ll \DL$.

The Stokes problem, as a saddle-point problem, is associated with the so-called inf-sup condition at both the continuous and discrete levels. It is well known that the continuous condition breaks down (i.e., the problem becomes ill-posed in the standard function space setting; see \Cref{sec:fa}) as the domain aspect ratio tends to infinity under certain boundary conditions. The most commonly studied case is that of Dirichlet boundary conditions, $\partial \Omega = \partial \Omega_D$ and $\fftwo =0$, where the Bogovskii operator allows for an explicit calculation. 
Specifically, it is known that the inf-sup constant $C(\Omega)$ in 
\begin{equation*}
\sup_{\uu\in \HH^1_0} \frac{(\nabla \cdot \uu, q)}{|\uu|_1} \ge \frac{1}{C(\Omega)}\|q\|_0, \quad q \in L^2_0 
\end{equation*}
depends on the domain aspect ratio $\DL/\DW$. For convex domains it scales directly as this aspect ratio, i.e., $C(\Omega) \sim \DL/\DW$, and tends to infinity as the domain is stretched; see, e.g., \cite{acosta2017divergence,dobrowolski2003lbb,   dobrowolski2005lbb, horgan1983inequalities,  olshanskii2000best, velte1996inequalities} for more details. For non-convex domains the scaling is potentially even worse; see, e.g., \cite[Lemma~3.1]{Galdi1998}.
Naively, one might expect that this breakdown of the inf-sup constant, and implied ill-posedness in the limit of stretched domains, points to the Stokes equation not being a good approximation of fluid flow on slender domains. However, the ubiquitous use of the Stokes equation on channel domains in engineering practice would contradict this. In this paper, we resolve this dilemma by showing that the Stokes equation is, in fact, well-posed on (possibly non-convex) slender domains, with constants independent of the aspect ratio, when using a weaker norm for the pressure variable.

The inf-sup constant also plays a crucial role in the design of preconditioners.
The condition number of the preconditioned matrix and the corresponding estimated number of iterations when using a Krylov method scales as
the inf-sup constant squared~\cite{rusten1992preconditioned, wathen1993fast}.  An example where the inf-sup constant is detrimental to the
performance of a standard preconditioner (based on well-posedness of \eqref{eq:stokes} in $\HH^1\times L^2$) was shown in, e.g., \cite{dobrowolski2012domain} for $2d$ channel domains or, more recently, in \cite{bertoluzza2017boundary}, \cite[Chap 3]{mri2fem2} in the context of $3d$ simulations
of blood flow and flow of the cerebrospinal fluid in the brain, respectively. 

The purpose of the present paper is threefold: 1) to unify previous mixed dimensional, domain decomposition and multilevel methods into a common practical framework; 2) to offer a continuous perspective in which the classical inf-sup condition is well behaved in slender geometries by exploiting a weaker norm than $L^2$ and 3) to cover cases with mixed boundary conditions which often occur in practical applications. In detail, previous preconditioning approaches have exploited centerline Laplacians for the pressure \cite{liu2021central,Meier2022}, while others have used domain decomposition or multilevel techniques involving coarse Laplacians
\cite{dobrowolski2012domain,sogn2023stable}. We will show here that a Laplacian is indeed beneficial and may be coarser or of lower dimensions if certain conditions are met. It is however crucial that the Laplacian is exploited in such a way that it weakens the $L^2$ norm as seen from the continuous perspective, which is achieved by considering an appropriate sum space. 
Our analysis here is closely related to \cite{dobrowolski2012domain} where homogeneous Dirichlet conditions were considered, but we extend the analysis both to the continuous case and to general boundary conditions typically appearing in applications.

The breakdown of the inf-sup condition occurs only for certain boundary conditions. It appears for the Dirichlet case, $\partial\Omega_D=\partial\Omega$,
where the pressure is only determined up to a constant, and hence, 
the inf-sup condition is stated in $\HH^1_0 \times L^2_0$. 
It also appears in the (perhaps more physically relevant) case of pipe flow: no-slip conditions (homogeneous Dirichlet conditions)
on the long boundaries (e.g., on the walls of the pipes) and stress/pressure (Neumann conditions) at the short boundaries (pipe inlet and outlet) \cite{braack2025infsupconditionstokesoutflow};
see also \Cref{ex:noslip,ex:freeslip} below.
In other words, the instability is due to Dirichlet conditions on the long edges (characteristic length scale $\DL$), which are separated by a small gap (characteristic
length scale $\DW$). See \cite[Section~3]{duran2011elementary} for a simple explanation in the Dirichlet case. On the other hand, the inf-sup constant appears independent of the aspect ratio for pure Neumann boundary conditions, and so this case will not be considered here. Specifically, we always assume that $\partial\Omega_D\neq\emptyset$. Moreover, to stay within the regime of a large aspect ratio, we also assume that the distance to a Dirichlet boundary from any point in the domain is at most some absolute constant times $\DW$. 

We remark that a similar dependence on the aspect ratio occurs for the constant in Korn's inequalities, which bounds the $\HH^1$ semi-norm $\uu$ by the Korn constant times the $\LL^2$ norm of $\nabla^S\uu$ (in the case of pure Neumann boundary conditions this requires $\uu$ being orthogonal to rigid body motions). However, this aspect ratio dependence does not occur for the boundary conditions considered in this paper (pure Dirichlet and pipe-flow); see \cref{ex:symgrad}.
Interestingly enough, the inf-sup constant $C(\Omega)$ in the case of pure Dirichlet boundary conditions can be directly related to the Korn constant in the case of pure Neumann boundary conditions; see \cite{horgan1983inequalities,velte1996inequalities} for details. Similarly, we would also expect the Korn constant to depend on the aspect ratio when Neumann conditions are prescribed on the long boundaries and no-slip conditions are prescribed on the short boundaries (physically a very short pipe). However, such boundary conditions seem artificial for pure fluid flow, and will also not lead to an aspect ratio dependence on the inf-sup constant, and they are not considered herein.

Our analysis of the Stokes system will yield a preconditioner of the form
  \begin{align}\label{eq:precond_espen}
 \mathcal{B} &= 
 \left(\begin{array}{cc}
      -\Delta^{-1} &  \\
       & I^{-1} + (-\alpha\DW^2\Delta)^{-1} 
 \end{array} \right)
\shortintertext{or alternatively}  
\label{eq:precond_manfred}
 \mathcal{B}_H &= 
 \left(\begin{array}{cc}
      -\Delta^{-1} &  \\
       & I^{-1} + (-\alpha\DW^2\Delta_H)^{-1} 
 \end{array} \right),
 \end{align}
 where $-\Delta_H$ is a Laplacian in a coarser space~\cite{dobrowolski2012domain} or  
$1d$ centerline \cite{liu2021central} and $\alpha$ is a constant that can be set equal to $1$. However, by considering a simplified cross-section of the domain, we motivate in \Cref{sec:coarse} why $\alpha\approx \tfrac{1}{12}$ provides a better choice for $2d$ channel domains. The effect of the constant $\alpha$  is further investigated numerically in \Cref{sec:numerics}.
We also stress that the coarse space length $H\approx W$ is \emph{not} related to any fine scale discretization parameter but rather characterizes a coarse partition of the domain where each element's diameter is close to the domain width $\DW$.
In \Cref{sec:theory}, we show that the above preconditioners are robust with respect to the domain aspect ratio.

An outline of the paper is as follows: \Cref{sec:prelim} and \Cref{sec:fa} introduce the
necessary notation and functional setting. We prove a domain robust inf-sup condition in \Cref{sec:theory}.
In \Cref{sec:coarse} we link the preconditioners to a dimensionally reduced model
of Stokes flow. Finally, numerical experiments demonstrating the robustness of the preconditioners are presented in \Cref{sec:numerics}.

\section{Preliminaries and notation}\label{sec:prelim}

Let $\Omega$ be a slender domain characterized by its length $\DL$ and width $\DW$, where $\DW \ll \DL$.
The space $L^2(\Omega)$ consists of square integrable functions, while $L^2_0(\Omega)$ is the subspace consisting of functions with mean value zero. 
Further, $H^1(\Omega)$ denotes the Sobolev space with first-order derivatives in $L^2(\Omega)$.
The domain boundary is denoted by $\partial \Omega = \partial \Omega_D \cup \partial \Omega_N$, where $\partial \Omega_D$ 
and $\partial \Omega_N$ are the parts of the boundary associated with Dirichlet and Neumann conditions, 
respectively. Let $H^1_{S,g}$ be the subspace of $H^1(\Omega)$ with trace equal to $g$ on $\partial \Omega_S$, while $H^1_0(\Omega)$ has zero trace on the complete boundary.
We let $(\cdot, \cdot)_\Omega$ denote the $L^2$ inner product and the corresponding duality product.
When it is clear from the presentation, we omit $\Omega$ from the above symbols.
We let $\|\cdot\|_0$ denote the $L^2$ norm, while 
$|\cdot|_1$ denotes the $H^1$ semi-norm involving only the first-order derivatives. 
Under the boundary conditions in this paper ($\Omega_D \neq \emptyset$),
the full $H^1$ norm and the semi-norm are equivalent, but this equivalence depends on the Poincaré constant for $\Omega$, which again depends on the diameter of the domain. Using the semi-norm thus allows for a more precise analysis. 
Boldface is used for vector-valued functions and spaces. The letter $C$ is reserved for generic constants, independent of the length $\DL$, the width $\DW$, and the aspect ratio $\DL/\DW$, that may change their value from occurrence to occurrence.

If $X$ and $Y$ are Hilbert spaces with dual spaces $X'$ and $Y'$, then $X \cap Y$ and $X+Y$ denote the intersection 
and sum spaces, respectively. 
Both $X\cap Y$ and $X+Y$ are Hilbert spaces. 
The {\itshape intersection norm} is defined as 
\begin{equation*}
 \|u\|^2_{X \cap Y} = \|u\|^2_X + \|u\|^2_Y ,
\end{equation*}
and, assuming $X\subset Y$, the {\itshape sum norm} is equal to 
\begin{equation*}
 \|u\|^2_{X + Y} = \inf_{u_1 \in X} \left\lbrace \|u_1\|^2_X + \|u - u_1\|^2_Y \right\rbrace. 
\end{equation*}

\begin{remark}
 Note that $(X\cap Y)' = X' + Y'$ and further
 that if $A:X\rightarrow X'$ and $B: Y\rightarrow Y'$ are both coercive  isomorphisms 
 then $A+B : X \cap Y \rightarrow X' + Y'$ 
 and $(A^{-1} + B^{-1})^{-1}: X + Y \rightarrow X' \cap Y'$. 
For a more detailed analysis of intersection and sum spaces, consider~\cite{bergh2012interpolation}.
We point out that~\cite{bergh2012interpolation} uses a slightly different definition of
 the intersection and sum space to the one above, which is more practical 
 for implementation, see~\cite{baerland2020observation} for an elaboration.
\end{remark}

A main motivation of the present paper is to construct efficient preconditioners
for Stokes flow in slender domains. We follow the theory of operator preconditioning
for saddle-point problems~\cite{mardal2011preconditioning}. That is, 
if 
$\mathcal A : V \times Q \rightarrow V' \times Q'$
corresponds to a  saddle-point problem of the form
\begin{equation*}
\mathcal{A}\left(\begin{array}{cc}
      u \\
      p
 \end{array} \right)
 = 
 \left(\begin{array}{cc}
      A & B' \\
      B &
 \end{array} \right)
  \left(\begin{array}{cc}
      u \\
      p
 \end{array} \right)
 = 
  \left(\begin{array}{cc}
      f \\
      g
 \end{array} \right),
\end{equation*}
with $A$ and $B$ satisfying the Brezzi conditions in $V$ and $Q$, we construct block preconditioners
$\mathcal B : V' \times Q' \to V \times Q$
of the form 
\begin{equation*}
 \mathcal{B} = 
 \left(\begin{array}{cc}
      N &  \\
       & M 
 \end{array} \right),
\end{equation*}
 with $N: V' \to V$ and $M: Q' \to Q$
 such that $\mathcal{B}$ is spectrally equivalent to $\mathcal{A}^{-1}$.
(In practice, both $N$ and $M$ are realized by order-optimal multi-level preconditioners
that are spectrally equivalent to Riesz operators in the discrete setting.)

In this work, we will encounter the particular case 
\begin{align*}
 &\mathcal{A}: V \times (Q_1 + Q_2)  \to  V' \times (Q_1' \cap Q_2')
 \\
  &\mathcal{B}: V' \times (Q'_1 \cap Q'_2)  \to  V \times (Q_1 + Q_2) .
\end{align*}
Since the dual space of the pressure space, $Q'_1 \cap Q'_2$, is  an intersection space, we can directly infer the block preconditioner
\begin{equation*}
 \mathcal{B} = 
 \left(\begin{array}{cc}
      N &  \\
       & M_1 + M_2 
 \end{array} \right),
\end{equation*}
where $N: V' \to V$, $M_1: Q_1' \to Q_1$, and 
$M_2: Q_2' \to Q_2$ are isomorphisms corresponding to the inverse of standard elliptic operators. 

\section{Weak form and functional setting}\label{sec:fa}
 The \emph{standard weak formulation} of the Stokes problem (using a lifting operator to ensure
 homogeneous Dirichlet and Neumann conditions, which simplifies the analysis) is: 
 Find $\uu \in \HH^1_{D, 0}$ and $p \in L^2$ such that 
 \begin{equation}
 \begin{aligned}
 (\nabla \uu, \nabla \vv) + (p, \nabla \cdot \vv) &= (\ffone, \vv), && \forall \vv \in  \HH^1_{D, 0},  \\
 (\nabla \cdot \uu, q)    &= 0, && \forall q \in L^2 . 
 \end{aligned}
 \label{mixed:stokes}
\end{equation}

\begin{remark}
If $\partial \Omega_D = \partial \Omega$, then $p \in L^2_0$ rather than $L^2$. We analyze both cases subsequently. 
 The lifting employed consists in letting $\uu = \hat \uu + \bar \uu$
 and $p = \hat p + \bar p$,
 where 
 $\hat \uu = \fftwo \text{ on } \partial \Omega_D$,
 $\frac{\partial \uu}{\partial \nn} = 0 \text{ on } \partial \Omega_N$,
 $\hat p = \ffthree \cdot \nn \text{ on } \partial \Omega_N$,
 and
 $\frac{\partial \uu}{\partial \nn} = \ffthree - \hat p \nn \text{ on } \partial \Omega_N$ and then solve for $\bar \uu$ and $\bar p$ with an alternative $\ffone$
 derived from subtracting $\hat \uu$ and $\hat p$ from \eqref{eq:stokes1}--\eqref{eq:stokes4}. Below $\uu$ and $p$ denotes the corresponding $\bar \uu$
 and $\bar p$. 
 \end{remark}

The breakdown of the inf-sup condition for the Stokes problem in slender domains is commonly studied under
no-slip boundary conditions \cite{dobrowolski2003lbb, dobrowolski2005lbb, dobrowolski2012domain, duran2011elementary}. Here, we illustrate that the issue also pertains to other boundary conditions.

\begin{example}[No-slip and traction boundaries]\label{ex:noslip}%
  We consider the Stokes problem \eqref{eq:stokes} with mixed boundary conditions on domain $\Omega = (0,\DL) \times (0,1)$ and two configurations
  of boundary conditions: (i) where we let $\partial\Omega_D$ be the top boundary and (ii) where $\partial\Omega_D$ are top and bottom edges.
  The remaining boundaries form $\partial\Omega_N$ where traction is prescribed. Recalling \Cref{sec:prelim}, for both cases the operator and
  the standard Stokes preconditioner \cite{elman2014finite} read
  \begin{equation}\label{eq:B_standard}
 \mathcal{A} = 
 \left(\begin{array}{cc}
      -\Delta & -\nabla \\
      \nabla\cdot & 
 \end{array} \right),
\quad
 \mathcal{B} = 
 \left(\begin{array}{cc}
      -\Delta^{-1} &  \\
                  & I^{-1}
 \end{array} \right).
 \end{equation}
  We shall discretize the problem with a discrete inf-sup stable finite elements, the lowest order Taylor-Hood
  $[\mathbb{P}_2]^2\times\mathbb{P}_1$ and non-conforming $[\mathbb{CR}_1]^2\times\mathbb{P}_0$ pair using Crouzeix-Raviart element and facet stabilization~\cite{burman2005stabilized} for the velocity\footnote{
  In the following $\mathbb{P}_k$, $k\geq1$ denotes the space of continuous piecewise polynomials of degree $k$ with $\mathbb{P}_0$ the space
  of element-wise constant functions. For verification of convergence of various Stokes discretizations employed in the manuscript, we refer to \Cref{sec:cvrg}.
  }. 
(In the latter case, the discretization of $-\Delta$ requires stabilization \cite{burman2005stabilized}.)

  For a given $L$, we consider a series of uniformly refined
  meshes with mesh size on the given level $l$ such that $h=h_l=h_02^{-l}$.
  A sample mesh for level $l=0$ is shown in \Cref{fig:coarse_mesh}.
  We assess performance of preconditioner \eqref{eq:B_standard} in terms of condition number\footnote{
  The condition number is estimated (here and in the rest of the manuscript if not noted otherwise) by iteratively computing the largest and smallest magnitude eigenvalues of the operator. To this end,
  we use the Krylov-Schur solvers from SLEPc \cite{hernandez2005slepc} and Spectra \cite{Spectra_v101}.
  } of the discrete preconditioned operator $\mathcal{B}_h\mathcal{A}_h$
  as well as the number of MinRes iterations required to solve\footnote{
  In all the examples below, the source terms and boundary conditions for the Stokes problem are manufactured based on
  the solution given in \eqref{eq:mms} unless stated otherwise.
  } \eqref{eq:stokes} with relative tolerance of $10^{-12}$ for the preconditioned
  residual norm. Here, the preconditioner is realized by LU decomposition.

  In \Cref{tab:B_standard_TTD} we observe that using \eqref{eq:B_standard} in case (ii) the condition number scales as $\DL^2$
  for both discretizations considered. The MinRes iterations increase with $\DL$ as well; however, the growth is less dramatic. On the
  other hand, when prescribing traction on a long boundary \eqref{eq:B_standard} is robust with respect to $L$.

  As stated earlier, the blow-up of the condition number in (ii) is caused by a breakdown of the inf-sup constant as $\DL$ increases.
  In \Cref{fig:standard_infsup_mode} we therefore plot the eigenmodes corresponding to the few smallest (in magnitude) eigenvalues of $\mathcal{B}_h\mathcal{A}_h$.
  Notably, the pressure component of the modes is smooth and varies mostly in the elongated direction. The flow appears compartmentalized. (If $x=(\uu, p)$ denotes the eigenvector of $\mathcal{A}x=\lambda \mathcal{B}^{-1}x$ for some $\lambda\neq 0$, it holds that $\nabla\cdot\uu=\lambda p$,
  i.e. $\uu$ is not divergence free.) 
\begin{table}
  \centering
\caption{Performance of standard Stokes preconditioner \eqref{eq:B_standard}
  for problem \eqref{eq:stokes} on domain $(0, L)\times (0, 1)$ with mixed boundary conditions, such that (long no-slip) velocity is prescribed
  on top and bottom edges, or (long no-slip and long traction) velocity is set on the top edge, and traction is set on the bottom edge.
  Condition number of the preconditioned problem together with MinRes iterations (in parentheses)
  is shown for different levels of mesh refinement $l$.
  }
\label{tab:B_standard_TTD}
\footnotesize
\setlength{\tabcolsep}{3pt}    {
  \begin{tabular}{c|llll||llll}
    \hline
    & \multicolumn{4}{c||}{$[\mathbb{P}_2]^2\times\mathbb{P}_1$} & \multicolumn{4}{c}{$[\mathbb{CR}_1]^2\times\mathbb{P}_0$}\\
    \hline
    \diagbox{$L$}{$l$} & 1 & 2 & 3 & 4 & 1 & 2 & 3 & 4\\
    \hline
    & \multicolumn{8}{c}{Long no-slip and long traction}\\
    \hline
1 & 5.83(45) & 5.86(46) & 5.88(46) & 5.92(46)    & 4.66(38) & 4.63(36) & 5.07(42) & 5.24(46)\\ 
3 & 5.69(45) & 5.79(46) & 5.83(46) & 5.85(46)    & 4.65(38) & 5.03(43) & 5.18(45) & 5.24(46)\\ 
5 & 5.69(46) & 5.79(47) & 5.83(46) & 5.85(46)    & 4.76(40) & 5.07(43) & 5.19(44) & 5.24(45)\\ 
10 & 5.72(46) & 5.79(46) & 5.83(46) & 5.85(47)   & 4.81(42) & 5.08(43) & 5.2(44) & 5.24(44)\\  
20 & 5.7(46) & 5.79(46) & 5.83(46) & 5.84(46)    & 4.82(42) & 5.09(43) & 5.2(44) & 5.24(44)\\  
50 & 5.73(46) & 5.79(45) & 5.83(46) & 5.85(46)   & 4.82(41) & 5.09(42) & 5.2(43) & 5.24(44)\\  
\hline
    & \multicolumn{8}{c}{Long no-slip}\\
    \hline
1 & 6.38(44) & 6.38(44) & 6.39(44) & 6.39(44)     & 6.3(39) & 6.3(40) & 6.39(42) & 6.42(44)\\
3 & 23.8(48) & 23.84(50) & 24.11(50) & 24.21(52)  & 23.51(46) & 23.55(47) & 24.14(52) & 24.32(54)\\
5 & 56.79(52) & 58.78(58) & 59.47(60) & 59.69(61) & 53.59(48) & 58.24(60) & 59.56(62) & 59.98(64)\\
10 & 214(74) & 222(78) & 224(82) & 225(86)        & 204(70) & 220(84) & 225(88) & 226(89)\\        
20 & 846(114) & 876(120) & 886(126) & 890(133)    & 808(114) & 869(128) & 888(134) & 894(136)\\    
50 & 5270(229) & 5455(240) & 5522(252) & 5540(256)& 5035(237) & 5418(256) & 5534(266) & 5567(270)\\
    \hline
  \end{tabular}}
\end{table}

\begin{figure}
  \centering
  \includegraphics[width=1.0\textwidth]{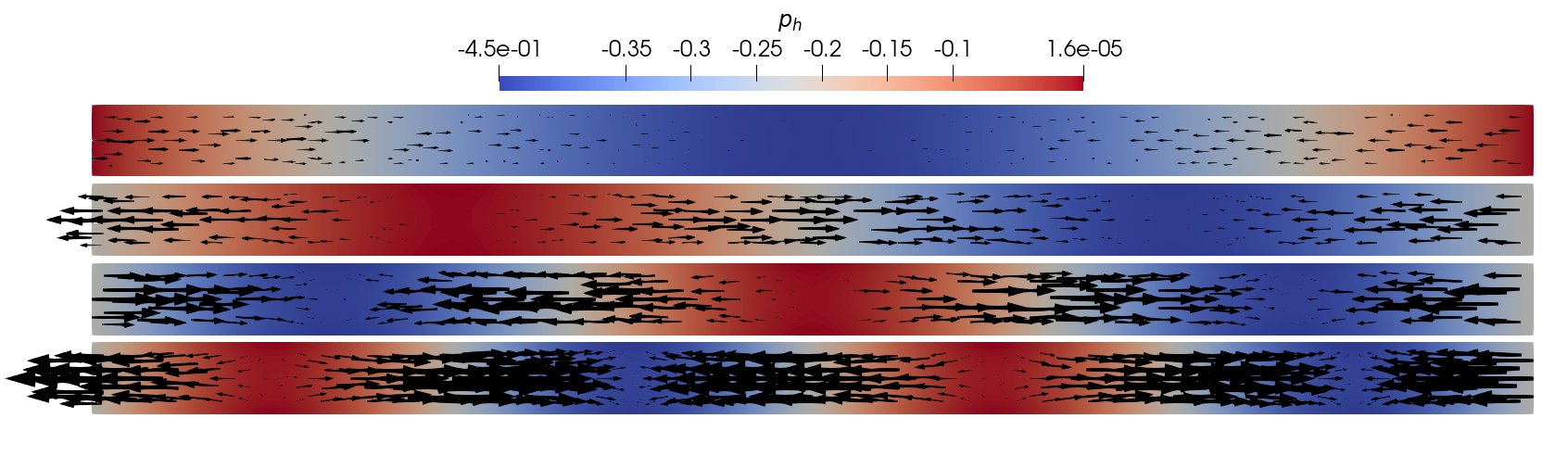}
  \vspace{-20pt}
  \caption{
    Eigenvalue problem for the Stokes problem on channel $L=20$ and mixed boundary conditions, where velocity is set on long boundaries, 
    using preconditioner \eqref{eq:B_standard}. Shown are the velocity (as glymphs) and pressure components
    of the eigenmodes corresponding to the first four smallest in magnitude eigenvalues; $\lvert \lambda_0 \rvert \approx 2\cdot10^{-3}$, $\lvert \lambda_1 \rvert \approx 8\cdot10^{-3}$,
    $\lvert \lambda_2 \rvert \approx 18\cdot10^{-2}$, $\lvert \lambda_3 \rvert \approx 31\cdot10^{-2}$.
    For the first mode, the inf-sup constant is attained.
  }
  \label{fig:standard_infsup_mode}
\end{figure}%
\end{example}

\begin{example}[Free-slip boundary]\label{ex:freeslip} 
  We consider the Stokes problem \eqref{eq:stokes} in $\Omega=(0, \DL)\times(0, 1)$ with an additional free-slip boundary
  $\partial\Omega_F=\partial\Omega\setminus(\partial\Omega_D\cup\partial\Omega_N)$ where we set
  \begin{equation}\label{eq:freeslip}
    \uu\cdot\nn = 0, \quad \nn^{\perp}\cdot\nabla\uu\cdot\nn=0\text{ on }\partial\Omega_F.
  \end{equation}
  Here $\nn^{\perp}$ is a unit tangent vector to $\partial\Omega_F$.
  This boundary condition is motivated by coupled multiphysics problems in which Stokes flow interacts across an interface with another medium, e.g., the air-water interface.
  In particular, the boundary conditions on $\partial\Omega_F$ reflect the coupling conditions on the interface; therein, the conservation of mass constrains
  only the normal component of the velocity and the condition $\nn^{\perp}\cdot\nabla\uu\cdot\nn=0$ allows the viscous
  fluid at the interface to slip freely in the tangential direction, cf. the Beavers-Joseph-Saffmann (BJS) condition \cite{saffman1971boundary},
 $\nn^{\perp}\cdot\nabla\uu\cdot\nn+\gamma \nn^{\perp}\cdot \uu = 0$,  where the fluid experiences friction.

  With $\Omega=(0, \DL)\times(0, 1)$, we prescribe traction on the left and right edges, and let (i)
  $\partial\Omega_F$, $\partial\Omega_D$ be respectively the top and bottom edges, or (ii) $\partial\Omega_D=\emptyset$. We remark
  that the case (ii) is well-posed in $\HH^1\setminus\textbf{RM}\times L^2$ where $\textbf{RM}=\left\{(1, 0)\right\}$ and the preconditioner component $-\Delta^{-1}$
  is then interpreted as pseudoinverse with respect to $\textbf{RM}$.
  Using preconditioner \eqref{eq:B_standard}, we see in \Cref{tab:B_standard_TTFD} that the condition
  number grows as $\DL^2$ in case (i). In contrast, when only the free-slip condition is applied on the long boundaries, the condition
  number shows only very small growth with $\DL$; $7.66$ for $\DL=1$ and $9.31$ for $\DL=50$.

  To summarize, \eqref{eq:B_standard} does not yield $\DL$-robust preconditioning
  for Stokes flow with mixed boundary conditions where no-slip and free-slip (or BJS condition) are prescribed on the opposing long boundaries.
  This is in contrast to the configuration with no-slip and traction boundary conditions in \Cref{ex:noslip}.
  Recalling the multiphysics setting that served as motivation for the current example, the sensitivity to domain length implies that the performance of
  recently proposed block-diagonal (material) parameter-robust preconditioners for coupled Stokes-Biot \cite{boon2022parameter} and
  Stokes-Darcy \cite{boon2022robust} problems may deteriorate on
  slender domains.

\begin{table}
  \centering
\caption{Performance of standard Stokes preconditioner \eqref{eq:B_standard}
  for the problem on the domain $(0, L)\times (0, 1)$ with mixed boundary conditions, where traction is set on the short left and right edges and
  free-slip and/or no-slip condition \eqref{eq:freeslip} is prescribed on long boundaries. Condition number of the preconditioned problem
  together with MinRes iterations (in parentheses)
  is shown for different levels of mesh refinement $l$. Discretization by Taylor-Hood $[\mathbb{P}_2]^2\times\mathbb{P}_1$ elements. Solver
  setup is identical to \Cref{ex:noslip}.
  }
\label{tab:B_standard_TTFD}
\footnotesize
\setlength{\tabcolsep}{3pt}    {
  \begin{tabular}{c|llll||llll}
    \hline
    & \multicolumn{4}{c||}{No-slip and free-slip} & \multicolumn{4}{c}{Free-slip only}\\
    \hline
    \diagbox{$L$}{$l$} & 1 & 2 & 3 & 4 & 1 & 2 & 3 & 4\\
    \hline
1 & 5.56(42) & 5.57(42) & 5.6(42) & 5.59(40)        & 7.66(35) & 7.66(35) & 7.66(35) & 7.66(35)\\ 
3 & 9.13(44) & 9.16(44) & 9.16(45) & 9.16(44)       & 7.55(35) & 7.56(35) & 7.56(33) & 7.56(35)\\ 
5 & 18.02(48) & 18.09(50) & 18.1(48) & 18.1(48)     & 7.64(35) & 7.64(35) & 7.65(35) & 7.65(31)\\ 
10 & 59.48(58) & 59.69(61) & 59.73(61) & 59.74(60)  & 7.85(35) & 7.85(35) & 7.85(35) & 7.85(35)\\ 
20 & 225(76) & 225(83) & 225(84) & 225(84)          & 8.25(35) & 8.25(35) & 8.25(35) & 8.25(35)\\ 
50 & 1383(137) & 1388(142) & 1389(150) & 1389(150)  & 9.30(35) & 9.31(35) & 9.31(35) & 9.31(34)\\  
\hline
  \end{tabular}}
\end{table}
\end{example}
 
As demonstrated in \Cref{ex:noslip,ex:freeslip}, the pressure norm in the formulation \eqref{mixed:stokes} is too strong and
causes a breakdown of the inf-sup condition in slender domains. We therefore propose a formulation based on a weaker pressure norm.   
The \emph{modified weak formulation} considered in this work is: 
 Find $\uu \in \HH^1_{D,0}$ and $p \in L^2+ \DW H^1_{N,0}$ such that 
\begin{equation}\label{eq:stokes_sumspace}
 \begin{aligned}
 (\nabla \uu, \nabla \vv) + (p, \nabla \cdot \vv) &= (\ff, \vv), && \forall \vv \in  \HH^1_{D,0}, \\
 (\nabla \cdot \uu, q)    &= 0, && \forall q \in L^2+ \DW H^1_{N,0} . 
 \end{aligned}
\end{equation}
If $\partial \Omega = \partial \Omega_D$, then 
 $p, q \in L^2_0 + \DW H^1 \cap L^2_0 $. 
 The first two of the four Brezzi conditions (see, e.g., \cite[Chapter~12.2]{Brenner2007-cg}) are the same as before. That is, both the coercivity and boundedness of 
$(\nabla \uu, \nabla \vv)$ are obvious in  $H^1_{D,0}$ as long as $\partial \Omega_D \not = \emptyset$ (as always assumed throughout this paper). 

The main purpose of this paper is to show that there is no dependence on the aspect ratio in the two remaining Brezzi conditions when replacing the $L^2$ norm by the weaker $L^2+ \DW H^1_{N,0}$ norm.
Specifically, for $q\in L^2(\Omega)$ let $\|\cdot\|_{*,\Omega}$ denote the inf-sup norm given by 
\begin{equation*}
\|q\|_{*}:=\sup_{\uu\in \HH^1_{0, D}(\Omega)} \frac{(\nabla\cdot \uu, q)}{|\uu|_{\HH^1(\Omega)}}.
\end{equation*}
(As before, when it is clear from the presentation, we will omit $\Omega$ from the above notation.
Moreover, we recall that in the case of pure Dirichlet boundary conditions, we have $q\in L^2_0(\Omega)$ and in this case it is common to use the notation $\|q\|_{*}=\|\grad q\|_{\HH^{-1}(\Omega)}$.)
The classical Brezzi conditions can be stated as
\begin{equation}\label{ineq:classical-Brezzi-conditions}
  \frac{1}{\sqrt{d}}\|q\|_{*}\leq \|q\|_{L^2} \leq C(\Omega)\|q\|_{*},  
\end{equation}
where we recall that $C(\Omega)$ is the inf-sup constant and $d$ is the dimension. 

Before continuing, let us relate the constant $C(\Omega)$ to the more familiar Poincaré constants. Define $C_{\mathrm{p}}^{\mathrm{D}}(\Omega)$ and $C_{\mathrm{p}}^{\mathrm{N}}(\Omega)$ 
to be the Poincaré constants corresponding to Dirichlet and Neumann boundary conditions, respectively.
They are defined as the smallest constants such that
\begin{align*}
    \|\phi\|_{L^2(\Omega)} &\leq C_{\mathrm{p}}^{\mathrm{D}}(\Omega)\|\nabla\phi\|_{L^2(\Omega)}, &&\forall\phi\in H^1_0(\Omega),
    \\
    \|\phi\|_{L^2(\Omega)} &\leq C_{\mathrm{p}}^{\mathrm{N}}(\Omega)\|\nabla\phi\|_{L^2(\Omega)}, &&\forall\phi\in H^1(\Omega)\cap L_0^2(\Omega).
\end{align*}
While these Poincaré constants are identical in one dimension (both equal to $\frac{L}{\pi}$ for an interval of length $L$), it is well-known that they behave quite differently for general domains in higher dimensions. The constant $C_{\mathrm{p}}^{\mathrm{D}}(\Omega)$ is well understood and one can show that $C_{\mathrm{p}}^{\mathrm{D}}(\Omega)\approx \DW$. On the other hand, the behavior of $C_{\mathrm{p}}^{\mathrm{N}}$ is to our knowledge not as well understood for general domains, however, for convex domains it has been shown in \cite{Payne:60} that $C_{\mathrm{p}}^{\mathrm{N}}(\Omega)\approx \DL$. 
A simple calculation now gives the following bound.
\begin{lemma}\label{lem:infsup-Poincare}
    The inf-sup constant $C(\Omega)$ satisfies
    \begin{equation*}
        \frac{1}{C(\Omega)} \leq \frac{C_{\mathrm{p}}^{\mathrm{D}}(\Omega)}{C_{\mathrm{p}}^{\mathrm{N}}(\Omega)}.
    \end{equation*}
\end{lemma}
\begin{proof}
By definition of the inf-sup constant, we have
\begin{equation*}
\begin{aligned}
    \frac{1}{C(\Omega)} 
    &= \inf_{q\in L_0^2(\Omega)}\sup_{\uu\in\HH^1_0(\Omega)} \frac{(q,\nabla\cdot \uu)}{\|q\|_{L^2(\Omega)}\|\nabla\uu\|_{L^2(\Omega)}}
    \\
    &\leq \inf_{q\in H^1(\Omega)\cap L_0^2(\Omega)}\sup_{\uu\in\HH^1_0(\Omega)}\frac{-(\nabla q,\uu)}{\|q\|_{L^2(\Omega)}\|\nabla\uu\|_{\LL^2(\Omega)}}
    \\
    &\leq C_{\mathrm{p}}^{\mathrm{D}}(\Omega)\inf_{q\in H^1(\Omega)\cap L_0^2(\Omega)}\frac{\|\nabla q\|_{L^2(\Omega)}}{\|q\|_{L^2(\Omega)}}
    \\
    &=\frac{C_{\mathrm{p}}^{\mathrm{D}}(\Omega)}{C_{\mathrm{p}}^{\mathrm{N}}(\Omega)},
\end{aligned}
\end{equation*}
where we have used the definition of $C_{\mathrm{p}}^{\mathrm{N}}(\Omega)$ in the last line.
\end{proof}
It follows from the above lemma and \cite{Payne:60} that $C(\Omega)\gtrsim \frac{\DL}{\DW}$ for convex domains. The other direction was shown in \cite{dobrowolski2003lbb,dobrowolski2005lbb}. A similar argument can be used for mixed boundary conditions as well. 

In this paper we will show that the norm $\|q\|_{*}$ is equivalent to $\|q\|_{L^2+ \DW H^1_{N,0}}$ in the case of mixed boundary conditions, and equivalent to $\|q\|_{L^2_0+ \DW H^1 \cap L^2_0}$ in the case of pure Dirichlet conditions, with equivalence constants independent of the aspect ratio of $\Omega$. 
These norms can be defined as
\begin{equation}\label{eq:sum-norm}
\begin{aligned}
    \|q\|_{L^2+\DW H^1_{N,0}} &= \inf_{\tilde{q}\in H^1_{N,0}}\sqrt{\|q-\tilde{q}\|^2_{L^2(\Omega)}+\DW^2|\tilde{q}|^2_{H^1(\Omega)}},
    \\
      \|q\|_{L^2_0+\DW H^1 \cap L^2_0} &= \inf_{\tilde{q}\in H^1 \cap L^2_0}\sqrt{\|q-\tilde{q}\|^2_{L^2(\Omega)}+\DW^2|\tilde{q}|^2_{H^1(\Omega)}}.  
    \end{aligned}
\end{equation}
In interpolation theory, norms of this type are referred to as $K$-functionals; see \cite[Chapter~3.1]{bergh2012interpolation} for more details about their properties.
Let us now look at the equivalence constants between the last sum-norm in \eqref{eq:sum-norm} and the $L^2$ norm on $L^2_0(\Omega)$. Since the sum-norm is weaker than the $L^2$ norm, it is straightforward to see that 
\begin{equation*}
     \|q\|_{L^2_0+\DW H^1 \cap L^2_0} = \inf_{\tilde{q}\in H^1 \cap L^2_0}\sqrt{\|q-\tilde{q}\|^2_{L^2(\Omega)}+\DW^2|\tilde{q}|^2_{H^1(\Omega)}}\leq \|q\|_{L^2}.
\end{equation*}
For the lower bound, we proceed as follows.
For any $\tilde{q}\in L^2_0(\Omega)\cap H^1(\Omega)$ we have
\begin{equation}
   \begin{aligned}\label{ineq:lower-bound}
    \|q\|_{L^2} &\leq \|q-\tilde{q}\|_{L^2} + \|\tilde{q}\|_{L^2}
    \\
    &\leq \|q-\tilde{q}\|_{L^2} + \frac{C_\mathrm{p}^\mathrm{N}(\Omega)}{\DW}\DW|\tilde{q}|_{H^1(\Omega)}
    \\
    &\leq \max\left\lbrace 1, \frac{C_\mathrm{p}^\mathrm{N}(\Omega)}{\DW}\right\rbrace\left(\|q-\tilde{q}\|_{L^2} +\DW|\tilde{q}|_{H^1(\Omega)}\right)
    \\
    &\leq \sqrt{2}\max\left\lbrace 1, \frac{C_\mathrm{p}^\mathrm{N}(\Omega)}{\DW}\right\rbrace\sqrt{\|q-\tilde{q}\|^2_{L^2(\Omega)}+\DW^2|\tilde{q}|^2_{H^1(\Omega)}}.
\end{aligned} 
\end{equation}
Since the above holds for any $\tilde{q}\in L^2_0(\Omega)\cap H^1(\Omega)$, it also holds for the infimum. Summarising, we have
\begin{equation}\label{ineq:equivalence-L2-sum}
    \|q\|_{L^2_0+\DW H^1 \cap L^2_0}\leq \|q\|_{L^2} \leq \sqrt{2}\max\left\lbrace 1, \frac{C_\mathrm{p}^\mathrm{N}(\Omega)}{\DW}\right\rbrace\|q\|_{L^2_0+\DW H^1 \cap L^2_0}.
\end{equation}
 Since the equivalence constants in \eqref{ineq:equivalence-L2-sum} behave in the same way as those between the $L^2$ norm and the $\|\cdot\|_*$ norm in \eqref{ineq:classical-Brezzi-conditions} (at least for convex domains),
then that is a good reason to expect that $\|\cdot\|_*$ is spectrally equivalent to $\|\cdot\|_{L^2_0+\DW H^1 \cap L^2_0}$, which is something we prove in the next section for general domains and both cases of boundary conditions in \eqref{eq:sum-norm}.
To provide a more intuitive understanding of the above sum space norm, we end this section with an illustrative example.

\begin{example}[The univariate case]
    Consider a long interval $(0,L)$. 
    In this case, a simple calculation shows that
    \begin{equation*}
        \|q'\|_{H^{-1}(0,L)}=\|q\|_{L^2(0,L)}, \quad \forall q\in L^2_0(0,L).
    \end{equation*}
    Since these norms are equal, we see that the scaling of the inf-sup constant is inherently a higher-dimensional problem, as one might already expect from it being bounded by the aspect ratio, and not the size of the domain. However, the behavior of the $L^2$ norm and the $H^1$ semi-norm on long intervals can still provide some insight into the norms studied in this paper. As there is no width of a one-dimensional interval, $\DW$ will in this example play the role of a free parameter (specifically the parameter of the $K$-functional) satisfying $\DW \ll \DL$. For any $q\in L^2_0(0,L)\cap H^1(0,L)$ we consider the standard cosine-expansion
        $q(x) = \sum_{k=1}^\infty c_k\phi_k(x),$
    where $\phi_k(x)=\sqrt{\frac{2}{L}}\cos\left(\frac{k\pi x}{L}\right)$ are orthonormal and span both $L^2_0(0,L)$ and $L^2_0(0,L)\cap H^1(0,L)$. By construction $\|\phi_k\|_{L^2(0,L)}=1$, and a simple calculation shows that $$\DW|\phi_k|_{H^1(0,L)}=\DW\|\phi_k'\|_{L^2(0,L)}=\frac{k\pi\DW}{L}.$$
    Thus for low frequencies, i.e., for $k$ satisfying $k\leq \frac{\DL}{\pi\DW}$, the weighted $H^1$ semi-norm is considerably weaker than the $L^2$ norm. In fact, if we let $n=\lfloor\frac{L}{\pi\DW}\rfloor$ and define $P_n$ to be the orthogonal projection onto $\operatorname{span}\{\phi_1,\ldots,\phi_n\}$, then a new norm on $L^2_0(0,L)$ given by
    \begin{equation*}
        \sqrt{\|q-P_nq\|^2_{L^2(0,L)}+\DW^2|P_nq|^2_{H^1(0,L)}}
    \end{equation*}
    would be significantly weaker than the standard $L^2$ norm for low frequencies. Consequently, this also holds for the sum norm of $L^2_0(0,L)+\DW H^1(0,L)\cap L^2_0(0,L)$ in \cref{eq:sum-norm}.
\end{example}

\section{The robust inf-sup condition}\label{sec:theory}
In this section, we begin by extending some of the results in \cite{dobrowolski2012domain} to include the case $\partial\Omega_D\neq\partial\Omega$. It is further necessary to repeat many of the arguments therein to keep better track of the involved constants, which allows us to use those results in the proof of our main theorem. As already mentioned in the introduction, to stay within the regime of a large aspect ratio, we assume that the Dirichlet boundary $\partial\Omega_D$ is associated with the length scale $\DL$. Moreover, the distance from any point in the domain to $\partial\Omega_D$ is assumed to be at most some absolute constant times $\DW$. 

As a first step in our argument, we partition the domain $\Omega$ into non-overlapping convex subdomains $\Omega_j, j=1,2,\ldots,N$, such that $\Omega=\bigcup_{j=1}^N \Omega_j$. This partition serves as a coarse mesh, but we stress that, contrary to \cite{dobrowolski2012domain}, it is primarily used as a theoretical tool and it is not necessary to construct it in practice. However, a coarse mesh \textit{can} be constructed, and we do so in some of our numerical examples. We let $\Hc_j=\operatorname{diam}\Omega_j$, $\Hc=\max_j\Hc_j$ and $\Hc_\text{min}=\min_j\Hc_j$, and consider the following properties of this theoretical mesh.
\begin{enumerate}[label=(\roman*)]
    \item \textit{Shape regularity}: $\Hc_j^d\leq C_{\text{shape}} |\Omega_j|$, for every $j=1,\ldots,N$.
    \label{shape-reg}
    \item \textit{Face regularity}: $\Hc_j^{d-1}\leq C_{\text{face}} |F|$, for every face $F\subset\partial \Omega_j$ and every $j=1,\ldots,N$.
    \label{face-reg}
    \item \textit{Mesh grading}: $\Hc_i\leq C_{\text{grade}} \Hc_j$, for every $i,j$ such that $\overline{\Omega_i}\cap\overline{\Omega_j}\neq\emptyset$.
    \label{grading}
    \item \textit{Coarseness}: $\DW \leq C_{\text{coarse}}\Hc_\text{min}$.
    \label{coarse}
\end{enumerate}
We note that shape/face regularity and mesh grading are typically assumed for a refinement procedure as the mesh size goes to zero, and where, crucially, the above constants are independent of the refinement level. Since we are only considering a fixed coarse partition, our mesh is trivially regular and graded in this typical sense. However, the constants we obtain for the norm equivalence of the preconditioner will depend directly on the above constants, as well as on the connectivity of the partition. In other words, we consider domains $\Omega$ for which there exist partitions $\{\Omega_j\}_j$ such that the connectivity is ``small'' and the constants in \ref{shape-reg}, \ref{face-reg}, \ref{grading}, and \ref{coarse} are ``reasonably close'' to $1$. Furthermore, observe that the constant in the inequality $\Hc\leq C\DW$ is dependent on $C_{\text{shape}}$. 

Let $Q_H$ denote the space of piecewise constant functions with respect to the partition $\{\Omega_j\}$.
By slightly abusing notation, we will use the same symbol $Q_H$ to refer to the $L^2$ projection onto the space $Q_H$ as well, i.e.,
\begin{equation}\label{def:Q}
(Q_H q)(x) = \frac{1}{|\Omega_i|}\int_{\Omega_i}q(y)dy,\quad x\in \Omega_i.
\end{equation}
We refer to the set of interior faces as $\mathcal{F}_\text{int}$, and for a given face $F$, we use the notation $\Omega_-$ and $\Omega_+$ to denote the two subdomains meeting at $F$. 
Moreover, we will let $H_F=\operatorname{diam} F$, and use the shorthand $C(F)=C(\Omega_-\cup\Omega_+)$ for the inf-sup constant. The set of faces on the Neumann boundary will be denoted by $\mathcal{F}_N$.

The next result is an extension of \cite[Lemma~2.1]{dobrowolski2012domain}.
\begin{lemma}\label{lem:infsup-norm}
    If $\Omega_j, j=1,\ldots N$ are disjoint open subsets of $\Omega$, then
    \begin{equation}
        \sum_{j=1}^N \|q\|^2_{*,\Omega_j} \leq \|q\|^2_{*,\Omega}, \quad \forall q\in Q,
    \end{equation}
    where $Q=L^2_0(\Omega)$ if $\partial\Omega_D=\partial\Omega$ and $Q=L^2(\Omega)$ otherwise.
\end{lemma}
\begin{proof}
    If $\partial\Omega_D=\partial\Omega$ then the result is proved in \cite[Lemma~2.1]{dobrowolski2012domain}. In the case of $\partial\Omega_D\neq\partial\Omega$, the argument has to be modified in the following way. Let $\ww_j\in \HH^1_{D,0}(\Omega_j)$, where we take no-slip conditions on all internal faces in addition to the no-slip Dirichlet boundary faces, and $\ww\in \HH^1_{D,0}(\Omega)$ be solutions of the problems
    \begin{align*}
        (\grad \ww_j,\grad\phiphi)_{\Omega_j} &= (q,\nabla\cdot \phiphi)_{\Omega_j}, &&\forall \phiphi\in \HH^1_{D,0}(\Omega_j),
        \\
        (\grad \ww,\grad\phiphi) &= (q,\nabla\cdot \phiphi), &&\forall \phiphi\in \HH^1_{D,0}(\Omega).
    \end{align*}
    Observe that
    \begin{equation}\label{eq:infsup-norm}
        \|q\|_{*,\Omega}=\sup_{\phiphi\in \HH^1_{0, D}(\Omega)} \frac{(q,\nabla\cdot \phiphi)}{|\phiphi|_{\HH^1(\Omega)}} 
        = \sup_{\phiphi\in \HH^1_{0, D}(\Omega)} \frac{(\grad \ww,\grad\phiphi)}{|\phiphi|_{\HH^1(\Omega)}} 
        = |\ww|_{\HH^1(\Omega)},
    \end{equation}
    and similarly $\|q\|_{*,\Omega_j}=|\ww_j|_{\HH^1(\Omega_j)}$ for $j=1,\ldots,N$.
    Extending the functions $\ww_j$ by $\mathbf{0}$ from all internal faces, we obtain
    \begin{align*}
        \sum_{j=1}^N \|q\|^2_{*,\Omega_j} &= \sum_{j=1}^N |\ww_j|^2_{\HH^1(\Omega_j)} = \sum_{j=1}^N (q,\nabla\cdot \ww_j)_{\Omega_j} = \sum_{j=1}^N (\grad \ww,\grad \ww_j)_{\Omega_j}
        \\
        &\leq \frac{1}{2}|\ww|^2_{\HH^1(\Omega)} + \frac{1}{2}\sum_{j=1}^N |\ww_j|^2_{\HH^1(\Omega_j)} = \frac{1}{2}|\ww|^2_{\HH^1(\Omega)} + \frac{1}{2}\sum_{j=1}^N \|q\|^2_{*,\Omega_j}
    \end{align*}
    and the result follows from \eqref{eq:infsup-norm}.
\end{proof}

\begin{lemma}\label{lem:Qh}
Let $Q_H$ be defined as in \eqref{def:Q}. Then
\begin{equation*}
\|q-Q_H q\|_{L^2(\Omega)}\leq \max_{j=1,\ldots,N}C(\Omega_j)\|q\|_{*,\Omega},
\end{equation*}
where $C(\Omega_j)$ is the inf-sup constant on $\Omega_j$.
\end{lemma}
\begin{proof}
This follows from local inf-sup conditions together with Lemma \ref{lem:infsup-norm}.
\end{proof}
Now, consider the discrete Laplacian on a coarse mesh, which is in general given by 
\begin{equation}\label{eq:laplace}
    (-\Delta_H p, q) = \sum_{j=1}^N \int_{\Omega_j} \nabla p\cdot \nabla q \mathrm{d}x+
\sum_{F\in\mathcal{F}_{\text{int}}} \int_{F} \frac{1}{H_F}\jump{p}_F\jump{q}_F\mathrm{d}s
+ \sum_{F\in\mathcal{F}_N} \int_{F} \frac{1}{H_F}p q\mathrm{d}s.
\end{equation}
Here $\jump{\cdot}_F$ denotes the jump at the face $F$. We also note that for $p, q\in Q_H$ the first term is zero.
 The standard discrete $H^1$ semi-norm on $q_H\in Q_H$ is then equal to
\begin{equation}\label{eq:DeltaH}
   |q_H|_{1,H}^2 = \langle (-\Delta_{H}) q_H, q_H\rangle 
   =\sum_{F\in\mathcal{F}_\text{int}} \int_{F} \frac{\jump{q_H}_F^2}{H_{F}}\mathrm{d}s + \sum_{F\in\mathcal{F}_N} \int_{F} \frac{q_H^2}{H_{F}}\mathrm{d}s.
\end{equation}
\begin{lemma}\label{lem:trace}
There is a constant $C>0$ such that for every $q_H\in Q_H$ we have
\begin{equation*}
    \|q_H\|_{*,\Omega} \leq C\Hc|q_H|_{1,H}.
\end{equation*}
\end{lemma}
\begin{proof}
Using the divergence theorem on each subdomain we get, for any $\vv\in\HH^1_{0, D}$,
\begin{align*}
    &(\nabla\cdot \vv, q_H) =\sum_{j=1}^N\int_{\Omega_j}\nabla\cdot \vv \, q_H \text{d}x
    \\
    &\leq \sum_{F\in\mathcal{F}_\text{int}}\int_{F}|\vv\cdot \nn \, [q_H]_{F}|\text{d}s +\sum_{F\in\mathcal{F}_N}\int_{F}|\vv\cdot \nn \, q_H|\text{d}s
    \\
    &\leq \left(\sum_{F\in\mathcal{F}_\text{int}\cup\mathcal{F}_N}\frac{\|\vv\|^2_{\LL^2(F)}}{\Hc_F}\right)^{1/2}\left(\sum_{F\in\mathcal{F}_\text{int}} \int_{F} \frac{\jump{\Hc_F q_H}^2}{H_{F}}\mathrm{d}s + \sum_{F\in\mathcal{F}_N} \int_{F} \frac{(\Hc_F q_H)^2}{\Hc_F}\mathrm{d}s\right)^{1/2}
    \\
    &\leq C \left(|\vv|^2_{\HH^1(\Omega)}+\frac{1}{H^2_{\text{min}}}\|\vv\|^2_{\LL^2(\Omega)}\right)^{1/2}|\Hc q_H|_{1,H}
    \\
    &\leq C\Hc|\vv|_{\HH^1(\Omega)}|q_H|_{1,H},
\end{align*}
where we have used the Cauchy-Schwarz inequality, the trace theorem (specifically, Young's inequality applied to \cite[Lemma~1.49]{Di_Pietro2011-tq}), and the Poincaré inequality. We observe that the integrals on the Dirichlet boundaries of $\Omega$ disappeared due to $\vv\in \HH^1_{0, D}$. Moreover, since we have assumed that Dirichlet boundaries are separated by the small distance $\DW$, the corresponding Poincaré constant behaves as $\DW$, and so the constant in the last line depends on \ref{coarse}, in addition to \ref{shape-reg} and \ref{face-reg} from the trace theorem. The result now follows from the definition of the norm $\|\cdot\|_*$.
\end{proof}

A key result is contained in the following lemma based on \cite[Lemma~3.3]{dobrowolski2012domain}. 
\begin{lemma}\label{lem:key}
Let $Q_H$ be the projection defined in \eqref{def:Q}. Then there exists a constant $C$ only depending on $C_{\text{shape}}$, $C_{\text{face}}$, $C_{\text{grade}}$ and the connectivity of the partition such that
\begin{equation*}
 \Hc|Q_H q|_{1,H}\leq C\max_{F\in\mathcal{F}_\text{int}\cup\mathcal{F}_N}C(F)\|q\|_{*,\Omega},
\end{equation*}
where $C(F)$ is the inf-sup constant on the union of the (at most two) subdomains sharing $F$ as a common face.
\end{lemma}
\begin{proof}
Consider two adjacent subdomains $\Omega_{-}$ and $\Omega_{+}$ sharing the common face $F\in \mathcal{F}_\text{int}$. We let $\coarsevol_\pm=|\Omega_\pm|$ and denote by $\overline{q}$ the mean value of $q$ on $\Omega_-\cup\Omega_+$. Then
\begin{equation*}
\overline{q} = \frac{q_-+q_+}{2\overline{\coarsevol}}, \quad \text{where} \quad q_{\pm}=\int_{\Omega_{\pm}}q \text{d}x, \quad \text{and} \quad 2\overline{\coarsevol} = \coarsevol_- + \coarsevol_+.
\end{equation*}
Following Dobrowolski's \cite{dobrowolski2012domain} calculation we find that
\begin{equation}\label{eq:dobrowolski_jump}
\begin{aligned}
\|q-\overline{q}\|^2_{L^2(\Omega_-\cup\Omega_+)} &= \int_{\Omega_-\cup\Omega_+}\left(|q|^2 - 2q\overline{q} + |\overline{q}|^2\right) dx
\\
&= \int_{\Omega_-}|q|^2dx + \int_{\Omega_+}|q|^2dx - \frac{|q_-+q_+|^2}{2\overline{\coarsevol}}
\\
&\geq \frac{|q_-|^2}{\coarsevol_-} + \frac{|q_+|^2}{\coarsevol_+}- \frac{|q_-+q_+|^2}{2\overline{\coarsevol}}
\\
&=\frac{\coarsevol_-\coarsevol_+}{2\overline{\coarsevol}}\left|\frac{q_+}{\coarsevol_+} - \frac{q_-}{\coarsevol_-}\right|^2
\\
&=\frac{\coarsevol_-\coarsevol_+}{2\overline{\coarsevol}}\jump{Q_H q}_F^2,
\end{aligned}
\end{equation}
It follows that
\begin{equation}\label{ineq:Omega_pm}
C\int_F\frac{\jump{\Hc_F Q_H q}_F^2}{\Hc_F}\text{d}s\leq \frac{\coarsevol_-\coarsevol_+}{2\overline{\coarsevol}}\jump{Q_H q}_F^2\leq \|q-\overline{q}\|^2_{L^2(\Omega_-\cup\Omega_+)}\leq C(F)^2\|q\|^2_{*,\Omega_-\cup\Omega_+}.
\end{equation}

For $F\in\mathcal{F}_N$, on the other hand, we let $\Omega_F$ denote its corresponding subdomain, and by using the Cauchy-Schwarz inequality on \eqref{def:Q}, we directly deduce that
\begin{equation}\label{ineq:Omega_Neumann}
\int_{F} \frac{(H_FQ_Hq)^2}{H_{F}}\mathrm{d}s= H_F|F| (Q_Hq)^2\leq \frac{H_F|F|}{|\Omega_F|} \|q\|^2_{L^2(\Omega_F)}
\leq C(F)^2C \|q\|^2_{*,\Omega_F}.
\end{equation}
Finally, since the subdomains $\Omega_-\cup\Omega_+$ in \eqref{ineq:Omega_pm} locally overlap finitely many times, we obtain our result by summing up \eqref{ineq:Omega_pm} and \eqref{ineq:Omega_Neumann}, taking square roots, and applying \Cref{lem:infsup-norm}.
\end{proof}

\begin{lemma}\label{lem:main-result-Dobrowolski}
    The norm $(\|q-Q_H q\|^2_{L^2(\Omega)} + \Hc^2|Q_H q|^2_{1,H})^{1/2}$ is equivalent to $\|q\|_{*,\Omega}$, with equivalence constants independent of the aspect ratio and size of the domain.
\end{lemma}
\begin{proof}
It follows from \Cref{lem:Qh} and \Cref{lem:key} that
\begin{equation*}
    \|q-Q_H q\|_{L^2(\Omega)} + \Hc|Q_H q|_{1,H} \leq C \|q\|_{*,\Omega}.
\end{equation*}
For the other direction, it follows from \Cref{lem:trace} that
   \begin{align*}
\|q\|_{*,\Omega} &\leq \|q-Q_Hq\|_{*,\Omega} + \|Q_Hq\|_{*,\Omega}
\\
&\leq C(\|q-Q_H q\|_{L^2(\Omega)} + \Hc|Q_H q|_{1,H}).
\end{align*} 
\end{proof}

\begin{theorem}\label{thm:main}
    The norm $\|q\|_{L^2+\DW H^1_{N,0}}$ is equivalent to $\|q\|_{*,\Omega}$, with equivalence constants independent of the aspect ratio and size of the domain.
\end{theorem}
\begin{proof}
We first construct a coarse mesh as described in this section and apply \Cref{lem:main-result-Dobrowolski}. Recall that  $\DW \leq C_{\text{coarse}}\Hc_\text{min}$ and $\Hc\leq C\DW$, where the last constant is dependent on $C_{\text{shape}}$.
Now, let $S_H:Q_H\to P^1_H$ be the Oswald quasi-interpolant \cite{oswald1993bpx}, which constructs a continuous piecewise linear function by averaging all nodal values. It satisfies the error and stability estimates
\begin{align*}
    \|q_H-S_Hq_H\|_{L^2}&\leq CH|q_H|_{1,H},
    \\
    \|\grad S_Hq_H\|_{L^2}&\leq C|q_H|_{1,H},
\end{align*}
see, e.g., \cite[Chapter~5.5.2]{Di_Pietro2011-tq}.
It then follows that
\begin{align*}
    \|q-S_HQ_Hq\|_{L^2}+\DW\|\grad S_HQ_Hq\|_{L^2} 
    &\leq \|q-Q_Hq\|_{L^2}+\|Q_Hq-S_HQ_Hq\|_{L^2}
    \\
    &\phantom{\leq} +\DW\|\grad S_HQ_Hq\|_{L^2}
    \\
    &\leq \|q-Q_Hq\|_{L^2}+CH|Q_Hq|_{1,H}
\shortintertext{and so}
   \|q\|_{L^2+\DW H^1}^2 &\leq C (\|q-Q_Hq\|_{L^2}^2+|HQ_Hq|_{1,H}^2).
\end{align*}

For the other direction, we observe that for any $\tilde{q}\in H^1_{N,0}(\Omega)$ (or $\tilde{q}\in H^1\cap L^2_0$ if $\partial \Omega_D = \partial \Omega$) we have
\begin{align*}
    \|q-Q_Hq\|_{L^2} &\leq \|q-\tilde{q}\|_{L^2}+ \|Q_H(\tilde{q}-q)\|_{L^2} + \|\tilde{q}-Q_H\tilde{q}\|_{L^2}
    \\
    &\leq 2\|q-\tilde{q}\|_{L^2} + CH\|\grad\tilde{q}\|_{L^2},
\shortintertext{and}
    |HQ_Hq|_{1,H}&\leq |HQ_H(q-\tilde{q})|_{1,H}+|HQ_H\tilde{q}|_{1,H}
    \\
    &\leq C\|Q_H(q-\tilde{q})\|_{L^2}+ C\|H\grad\tilde{q}\|_{L^2}
    \\
    &\leq C(\|q-\tilde{q}\|_{L^2}+ \|H\grad\tilde{q}\|_{L^2}),
\end{align*}
where we have used the inverse inequality and the stability of the $L^2$-projection in $H^1_{N,0}$: $|Q_H\tilde{q}|_{1,H}\leq C \|\grad\tilde{q}\|_{L^2}$.\footnote{Since $q\in H^1_{N,0}$ has zero jumps on the faces, this stability estimate follows from the trace theorem \cite[Lemma~1.49]{Di_Pietro2011-tq} combined with the local error estimate in \cite[Lemma~1.58]{Di_Pietro2011-tq}.}
Thus,
\begin{equation*}
    \|q-Q_Hq\|_{L^2}^2+|HQ_Hq|_{1,H}^2 \leq C(\|q-\tilde{q}\|_{L^2}^2+\DW^2\|\grad\tilde{q}\|_{L^2}^2)
\end{equation*}
for any $\tilde{q}\in H^1_{N,0}(\Omega)$, and so it is also true for the infimum in \eqref{eq:sum-norm}.
\end{proof}

\begin{remark}\label{rmrk:varW}
    Note that a more careful analysis using the norm 
    \begin{equation*}
    \|q-Q_Hq\|^2+\sum_{F}|H_FQ_Hq|_{1,H_F}^2,
    \end{equation*}
    where 
        $|q_H|_{1,H_F} = \int_{F} \frac{\jump{q_H}^2}{H_{F}}\mathrm{d}s,$
    would lead to norm equivalence with the more general sum-norm 
    \begin{equation}\label{eq:varW}
        \inf_{\tilde{q}\in H^1(\Omega)}\sqrt{\|q-\tilde{q}\|^2_{L^2(\Omega)}+|\DW\tilde{q}|^2_{H^1(\Omega)}}
    \end{equation}
    in \Cref{thm:main}, where $\DW(\mathbf{x})$ is now a function describing a locally shortest diameter. This would lead to the more general term $\nabla\cdot(\DW(\cdot)^2\grad)$
    instead of $\DW^2\Delta$ in our preconditioner \eqref{eq:precond_espen}.
\end{remark}

\begin{remark}
It follows from \cref{thm:main} that $\|q\|_{L^2+\DW H^1_{N,0}}$ should be a natural norm to use for the pressure variable in the Stokes equation on slender domains, if the goal is to obtain error estimates independent of the aspect ratio. Of course, changing the norm used to measure the error will not change the approximation itself (unless the norm is used in an adaptive method), so it is unclear whether there is any gain in using this norm for error estimates instead of the standard $L^2$ norm.
\end{remark}

\section{Model order reduction of Stokes in thin channels}\label{sec:coarse}
\begin{figure}[!htb]
    \centering
    \includegraphics[width=0.8\linewidth]{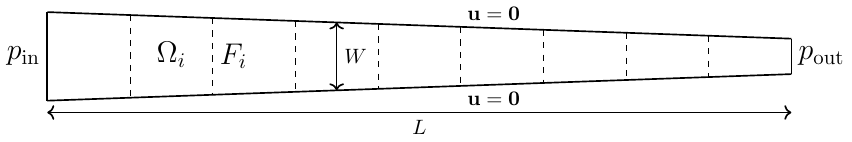}  
    \caption{Pressure-driven Stokes flow scenario in a simple channel domain (large aspect ratio, $L/W \gg 1$, small angle) discretized by coarse grid discretization with elements $\Omega_i$ (approximate aspect ratio $1$) and element faces $F_i$.}
    \label{fig:channel}
\end{figure}
The goal of this section is to estimate the prefactor $\alpha$ in \eqref{eq:precond_espen} and \eqref{eq:precond_manfred}. To understand Stokes flow in thin domains, it is instructive to show how the equations simplify for large aspect ratios. In this section, we explore a link between the suggested coarse-grid pressure operator and a dimensionally reduced Stokes problem.
Consider the dimensional Stokes equations in 3$d$ Cartesian coordinates $(x,y,z)$ where the z-dimension (length scale $\DW$) is much smaller than the other dimensions $ \mathbf{x}_\perp = (x,y)^T$ (length scale $\DL$). We split the equations into a momentum balance in the short direction and the remaining directions and non-dimensionalize using two different length scales: $\mathbf{x}_\perp = \mathbf{x}_\perp^\star \DL$, and $z = z^\star \DW$. We will consider the case of large aspect ratios $\epsilon:= \DW/\DL \ll 1$. With a characteristic velocity scale $\uu_\perp = U \uu_\perp^\star$, $u_z = \epsilon U u_z^\star$ and pressure scale $p = P p^\star$. In principally pressure-driven flow, we expect a pressure scale of $P = \frac{\mu U}{\epsilon D} =\frac{\mu U}{L \epsilon^2}$ (the typical pressure scale in lubrication theory~\cite{Batchelor2000}, $\mu$ denoting dynamic fluid viscosity) with longitudinal pressure gradient being balanced by the dominating viscous stresses. The non-dimensionalized equations in operator notation (with variables split into the component in the short direction and its complement) are
\begin{align}
\begin{pmatrix}
\epsilon^2 \Delta_\perp^{\star} + \partial_{z^\star}^2 & 0 & -\nabla_\perp^\star \\
0 & \epsilon^4 \Delta_\perp^{\star} + \epsilon^2\partial_{z^\star}^2 & -\partial_{z^\star} \\
\nabla_\perp^\star \cdot & \partial_{z^\star} & 0
\end{pmatrix}
\begin{pmatrix}
\uu_\perp^\star \\
u_z^\star \\
p^\star
\end{pmatrix}
=
\begin{pmatrix}
0 \\
0 \\
0
\end{pmatrix}.
\end{align}
This is approximately equal (up to an error $\mathcal{O}(\epsilon^2)$) to solving the simplified equations
\begin{align}
\begin{pmatrix}
\partial_{z^\star}^2 & 0 & -\nabla_\perp^\star \\
0 & 0 & -\partial_{z^\star} \\
\nabla_\perp^\star \cdot & \partial_{z^\star} & 0
\end{pmatrix}
\begin{pmatrix}
\uu_\perp^\star \\
u_z^\star \\
p^\star
\end{pmatrix}
=
\begin{pmatrix}
0 \\
0 \\
0
\end{pmatrix}.
\label{eq:reduced}
\end{align}

Now we consider a straight channel of height $W$ and length $L$. (In fact, since $\epsilon \ll 1$, the channel height may also vary slowly in $\mathbf{x}$, $W = W(\mathbf{x})$, without sacrificing the quality of approximation made above.)
Going back to dimensional quantities,
integrating the tangential momentum balance (twice) over $z$ with boundary conditions $\uu = \mathbf{0}$ (no slip and no penetration) for $z=0$ and $z=\DW$ yields an explicit solution for $\uu_\perp$ in terms of $p$:
\begin{equation}
\uu_\perp = -\frac{\DW^2}{12 \mu} \nabla_\perp p.
\label{eq:parplate}
\end{equation}
Moreover, we see from \eqref{eq:reduced} that $p$ does not vary with $z$ at fixed $\mathbf{x}$.
This observation also motivates a partition (akin to the partitioning used in our proof above) of the channel domain in a way that elements are channel sections (full height $W$) with an aspect ratio close to $1$; see \cref{fig:channel}.

Let us denote partition elements by $\Omega_i$ and let $q$ be a piecewise constant function with respect to the partitions $\Omega_i$, with $q_i$ the value of $q$ on $\Omega_i$. Moreover, let $\mathcal{N}_{i}$ be the set of elements $\Omega_j$ that share a face with $\Omega_i$, $\sigma_{ij}$ is the face shared by $\Omega_i$ and $\Omega_j$, and $\mathbf{n}_{\sigma_{ij}}$ denote a unit normal vector on $\sigma_{ij}$ pointing towards $\Omega_j$. Integration of the mass balance equation over the channel domain $\Omega$, using \eqref{eq:parplate}, and the no-slip, no-penetration boundary conditions on $z = 0$ and $z = W$ gives, similar to the proof of \cref{lem:trace}, that:
\begin{equation}
\label{eq:connection}
\begin{split}
\int_\Omega \nabla \cdot \uu \, q\text{d}x &= \sum\limits_{i=0}^N \int_{\Omega_i} \nabla \cdot \uu \, q_i\text{d}x = \sum\limits_{i=0}^N \int_{\partial\Omega_i} \uu \cdot \nn \, q_i\text{d}s \\
&= \sum\limits_{i=0}^N \sum\limits_{\Omega_j \in \mathcal{N}_{i}}\int_{\sigma_{ij}} \uu \cdot \mathbf{n}_{\sigma_{ij}} \, q_i \text{d}s + \sum\limits_{F \in \mathcal{F}_{N}} \int_{F} \uu \cdot \mathbf{n}_{F} q \,\text{d}s \\ 
&= \sum\limits_{F \in \mathcal{F}_\text{int}} \int_{F} \uu \cdot \mathbf{n}_{F} \, \jump{q}_F \text{d}s + \sum\limits_{F \in \mathcal{F}_{N}} \int_{F} \uu \cdot \mathbf{n}_{F} q \,\text{d}s \\ 
&= \sum\limits_{F \in \mathcal{F}_\text{int}} \int_{F} -\frac{\DW^2}{12 \mu} \nabla_\perp p \cdot \mathbf{n}_{F} \jump{q}_F \,\text{d}s
+ \sum\limits_{F \in \mathcal{F}_{N}} \int_{F} -\frac{\DW^2}{12 \mu} \nabla_\perp p  \cdot \mathbf{n}_{F} \, q\text{d}s\\
&\approx \sum\limits_{F \in \mathcal{F}_\text{int}} \int_{F}\frac{\DW^2}{12 \mu} \frac{1}{\ell_{F}} \jump{p}_F \jump{q}_F \,\text{d}s
+ \sum\limits_{F \in \mathcal{F}_{N}} \int_{F} \frac{\DW^2}{12 \mu}\frac{1}{\ell_{F}} \jump{p}_{F} \, q\text{d}s \quad \forall q.
\end{split}
\end{equation}
Here, we used mass conservation ($\uu \cdot \nn_{\sigma_{ij}} + \uu \cdot \nn_{\sigma_{ji}} = 0$) to summarize the face integrals over coinciding faces $\sigma_{ij}$ and $\sigma_{ji}$ (or interior face $F$), noting that for any $F$ (and normal $\nn_{F} \equiv \nn_{\sigma_{ij}}$),
\begin{equation*}
\begin{split}
    \uu \cdot \nn_{\sigma_{ij}} q_i + \uu \cdot \nn_{\sigma_{ji}} q_j &= (\uu \cdot \nn_{\sigma_{ij}} + \uu \cdot \nn_{\sigma_{ji}}) \frac{q_i + q_j}{2} + \frac{\uu \cdot \nn_{\sigma_{ij}} - \uu \cdot \nn_{\sigma_{ji}}}{2} (q_i - q_j)  \\
    &= \uu \cdot \nn_{F} \jump{q}_F,
\end{split}
\end{equation*}
with $\jump{q}_F = q_i - q_j$.
The approximation of the face integrals in \eqref{eq:connection} is based on a simple two-point flux approximation ($\nabla_\perp p \cdot \mathbf{n}_{\sigma_{ij}} \approx -\jump{p}/\ell_{F}$) where $\ell_{F}$ denotes the distance between the centroids of the neighboring subdomains $\Omega_i$ and $\Omega_j$ sharing face $F$. On the boundary, $\ell_{F}$ is the distance from the centroid of $\Omega_i$ to the centroid of the boundary face $F$, and $\jump{p}_{F} = p_i - p_F$ denotes the difference between boundary data $p_F$ and element average pressure $p_i$.

With $\ell_F \approx H_F$ (by ``shape regularity" of the coarse partition, as discussed in \Cref{sec:theory}), we observe that the expression in the last row of \eqref{eq:connection} matches \eqref{eq:laplace} up to prefactor $\tfrac{1}{12}$ (and up to the second-order approximation for the boundary terms). (Note that in formulation \eqref{eq:laplace}, the constant $\mu$ is absorbed into the pressure following our initial problem statement \eqref{eq:stokes}.)

\begin{remark} While the prefactor $\tfrac{1}{12}$ appears for channel domains of the ``parallel plates''-type (which we will use in the numerical examples in \Cref{sec:numerics}) prefactors for other cross-sectional geometries can be readily derived or computed, see e.g. \cite[Chap 3.3]{white2006viscous}. If the short directions consist of a cross-sectional plane $S$ of any shape, then the prefactor can be estimated as 
\begin{equation}\label{eq:prefactor}
    \frac{1}{W^2|S|}\int_S\! u\, \text{d}s,
\end{equation}
where $u$ solves $-\Delta_S u = 1$ in $S$ with $u=0$ on $\partial S$. 
\end{remark}
\begin{remark}
As a continuation of the previous remark, it is interesting to note that \eqref{eq:prefactor} can be directly related to the Poincaré constant $C_\mathrm{p}^\mathrm{D}(S)$ as follows. First use integration by parts to find that
\begin{equation}\label{eq:integrate-u}
   \int_S u\, \text{d}s= -\int_S u \Delta_S u \, \text{d}s= \int_S |\nabla_S u|^2\, \text{d}s.
\end{equation}
By combining \eqref{eq:integrate-u} with the Cauchy-Schwarz inequality and the Poincaré inequality, we get
\begin{equation*}
    \left(\int_S u\, \text{d}s\right)^2\leq |S| \int_S u^2\, \text{d}s \leq |S| C_\mathrm{p}^\mathrm{D}(S)^2 \int_S |\nabla_S u|^2\, \text{d}s = |S| C_\mathrm{p}^\mathrm{D}(S)^2 \left|\int_S u\, \text{d}s\right|,
\end{equation*}
which leads to the upper bound
\begin{equation*}
    \frac{1}{W^2|S|}\left|\int_S u\, \text{d}s\right| \leq \frac{C_\mathrm{p}^\mathrm{D}(S)^2}{W^2}.
\end{equation*}
For a $1d$ cross-section, we have that $C_\mathrm{p}^\mathrm{D}(S)=\tfrac{W}{\pi}$, and in this case we get
\begin{equation*}
    \frac{1}{W^2|S|}\left|\int_S u\, \text{d}s\right| \leq \frac{1}{\pi^2},
\end{equation*}
which is not too far off from $\tfrac{1}{12}$. Consequently, and as we might already expect from \cref{lem:infsup-Poincare}, a (close to) sharp Poincaré constant for the cross-section can provide a useful estimate of the prefactor $\alpha$. We also verify this numerically in the next section.
\end{remark}
We will investigate an optimal prefactor numerically in the next section and show that the above derivation provides relevant insight into the optimal constant, which the proof of \Cref{thm:main} does not capture exactly.

\section{Numerical experiments}\label{sec:numerics}
Here we present numerical experiments that complement theoretical analysis from previous sections and,
in particular, demonstrate domain robustness of the preconditioners \eqref{eq:precond_espen} and \eqref{eq:precond_manfred}. In all examples, condition numbers reported are estimated for the preconditioned operator ($\mathcal{B}_h\mathcal{A}_h$) using iterative methods. We first revisit the channel flow problems from \Cref{ex:noslip,ex:freeslip}.
Experiments with finite elements and finite volumes are implemented based on the FEniCS~\cite{logg2012automated} and DUNE/DuMu\textsuperscript{x}~\cite{Dune210,Koch2021} libraries, respectively.

\begin{example}[Channel flow]\label{ex:channels}
  We apply preconditioners \eqref{eq:precond_espen}, \eqref{eq:precond_manfred} to the channel $(0, \DL)\times(0, 1)$ flow problems
  setup in \Cref{ex:noslip,ex:freeslip} where the performance of standard preconditioner \eqref{eq:B_standard} deteriorated 
  when $L\gg 1$. Here we pick $\alpha=1$.
  
  For the preconditioner $\mathcal{B}_H$ the space $Q_H$ and
  operator $-\Delta_H$ are discretized on the quadrilateral mesh with $L$ square cells irrespective of the mesh size in $\Omega_h$.
    We recall that the pressure Laplacians in the preconditioners are defined in \eqref{eq:laplace} and, in particular, the
    boundary conditions on $\Gamma_N$ are enforced weakly.
        As we shall demonstrate later, this way of imposing the boundary conditions is preferable
        even if the chosen FE space is amenable to enforcing them strongly, e.g. for continuous  $\mathbb{P}_1$ elements.
    If $\lvert \Gamma_N \rvert=0$
    the Laplacians are singular and $(-\Delta)^{-1}$ is then understood as the pseudoinverse.
    Finally, in the application of the preconditioner $\mathcal{B}_H$ the mapping between
    $Q_h$ and $Q_H$ is realized by $L^2$-projection. In particular, a degree of freedom
    in $Q_H$ is determined by taking the mean value of functions in $Q_h$ over the corresponding
    quadrilateral coarse cell, see \Cref{fig:coarse_mesh} and \eqref{def:Q}.
  
\begin{figure}[!htb]
  \centering
  \includegraphics[width=\textwidth]{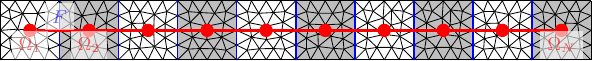}
  \vspace{-15pt}
  \caption{
  Coarse partitioning of domain $\Omega=(0, 10)\times (0, 1)$, cf. \Cref{sec:theory}, in terms of quadrilateral 
  mesh $\Omega_H$ with $1\times 1$ cells $\Omega_i$, $1\leq i \leq 10$. Operator $-\Delta_H$ in \eqref{eq:precond_manfred} is discretized on the space $Q_H$
  of piecewise constant functions with respect to $\Omega_H$ 
  leading to a ``one-dimensional'' Laplacian (with its ``segment'' mesh in red). Pressure space $Q_h$ is defined over the triangular mesh. Here the
  mesh for refinement level $l=0$, cf. \Cref{ex:noslip}, is shown.
  }
  \label{fig:coarse_mesh}
\end{figure}

  Applying discretization with Taylor-Hood elements and using the same solver setup as in \Cref{ex:noslip,ex:freeslip}
  we observe in \Cref{tab:espen_manfred_TD} and \Cref{tab:espen_manfred_FD} that both preconditioners \eqref{eq:precond_espen}, \eqref{eq:precond_manfred}
  yield bounded condition numbers and MinRes iterations as the aspect ratio of the domain increases. We note that the preconditioners are very similar in their performance. Notably, with $\mathcal{B}_H$ we have $\text{dim}Q_H=L$, i.e., the cost
  of assembly and inversion of operator $-\Delta_H$ is constant with mesh refinement, cf. $-\Delta$ on $Q_h$ in $\mathcal{B}$. However,
  $\mathcal{B}$ can be more practical/straightforward for implementation as no additional discretization is required.
\begin{table}
  \centering
  \caption{Performance of preconditioners \eqref{eq:precond_espen}, \eqref{eq:precond_manfred}
    for the problem on the domain $(0, L)\times (0, 1)$ with mixed boundary conditions where traction is set on short left and right edges and
    no-slip condition is prescribed on long edges, cf. \Cref{tab:B_standard_TTD} for results with standard preconditioner \eqref{eq:B_standard}. Condition number
    of the preconditioned problem
    together with MinRes iterations (in parentheses)
    is shown for different levels of mesh refinement $l$. Discretization by Taylor-Hood $[\mathbb{P}_2]^2\times\mathbb{P}_1$ elements.
    Preconditioner $\mathcal{B}_H$ uses mesh $\Omega_H$ of $1\times 1$ quadrilateral cells, cf. \Cref{fig:coarse_mesh}.
    Solver setup is identical to \Cref{ex:noslip}.
  }
\label{tab:espen_manfred_TD}
\footnotesize
\setlength{\tabcolsep}{3pt}    {
  \begin{tabular}{c|llll||llll}
    \hline
    & \multicolumn{4}{c||}{$\mathcal{B}$} & \multicolumn{4}{c}{$\mathcal{B}_H$}\\
    \hline
    \diagbox{$L$}{$l$} & 1 & 2 & 3 & 4 & 1 & 2 & 3 & 4\\
    \hline
3 & 11.75(49) & 11.77(50) & 12.54(51) & 12.95(52)    & 13.69(50) & 13.71(50) & 13.9(51) & 13.98(52)\\ 
5 & 14.56(48) & 16.06(56) & 17(57) & 17.51(58)       & 17.91(49) & 18.55(56) & 18.92(56) & 19.1(59)\\ 
10 & 19.04(67) & 20.28(74) & 21(77) & 21.4(80)       & 21.38(70) & 22.35(74) & 22.97(78) & 23.29(83)\\
20 & 21.36(84) & 22.23(87) & 22.66(92) & 22.88(95)   & 22.45(82) & 23.57(91) & 24.32(91) & 24.72(97)\\
50 & 22.37(88) & 23.08(92) & 23.33(97) & 23.42(95)   & 22.72(86) & 23.91(91) & 24.73(97) & 25.15(99)\\
100 & 22.55(86) & 23.24(91) & 23.46(93) & 23.53(95)  & 22.75(86) & 23.94(89) & 24.76(93) & 25.21(97)\\    
\hline
  \end{tabular}}
\end{table}

\begin{table}
  \centering
\caption{Performance of preconditioners \eqref{eq:precond_espen}, \eqref{eq:precond_manfred}
  for problem on domain $(0, L)\times (0, 1)$ with mixed boundary conditions, no-slip on bottom edge, free-slip \eqref{eq:freeslip} on top edge
  and traction on the short lateral edges. For this setup, the standard preconditioner \eqref{eq:B_standard} fails, as shown in \Cref{ex:freeslip}.
    Condition number
    of the preconditioned problem
    together with MinRes iterations (in parentheses)
    is shown for different levels of mesh refinement $l$. Discretization by Taylor-Hood $[\mathbb{P}_2]^2\times\mathbb{P}_1$ elements.
    Preconditioner $\mathcal{B}_H$ uses mesh $\Omega_H$ of $1\times 1$ quadrilateral cells, cf. \Cref{fig:coarse_mesh}.
    Solver setup is identical to \Cref{ex:noslip}.}
\label{tab:espen_manfred_FD}
\footnotesize
\setlength{\tabcolsep}{3pt}    {
  \begin{tabular}{c|llll||llll}
    \hline
    & \multicolumn{4}{c||}{$\mathcal{B}$} & \multicolumn{4}{c}{$\mathcal{B}_H$}\\
    \hline
    \diagbox{$L$}{$l$} & 1 & 2 & 3 & 4 & 1 & 2 & 3 & 4\\
    \hline
1 & 5.63(45) & 5.64(45) & 5.65(45) & 5.68(46)  &  5.59(44) & 5.61(45) & 5.62(45) & 5.65(46)\\
3 & 7.34(45) & 7.34(45) & 7.33(45) & 7.32(47)  &  7.28(44) & 7.29(45) & 7.29(45) & 7.29(48)\\
5 & 7.34(45) & 7.34(46) & 7.33(46) & 7.32(49)  &  7.28(45) & 7.29(45) & 7.29(45) & 7.29(49)\\
10 & 7.34(47) & 7.34(48) & 7.33(51) & 7.32(51) &  7.28(45) & 7.29(45) & 7.29(47) & 7.29(48)\\
20 & 7.34(48) & 7.34(48) & 7.33(51) & 7.32(51) &  7.28(45) & 7.29(45) & 7.29(47) & 7.29(47)\\
50 & 7.34(49) & 7.34(50) & 7.33(52) & 7.32(51) &  7.28(45) & 7.29(45) & 7.29(47) & 7.29(47)\\
100 & 7.34(49) & 7.34(49) & 7.33(51) & 7.32(51)&  7.28(45) & 7.29(44) & 7.29(47) & 7.29(47)\\    
\hline
  \end{tabular}}
\end{table}
\end{example}

For the flow in domains with variable thickness \Cref{rmrk:varW} notes on a stable inf-sup condition
with respect to the pressure space $L^2+WH^1$ where the weight $W$ varies in space describing
the local smallest diameter of the domain. The refined analysis yields preconditioners of the form
\begin{align}\label{eq:precond_varW}
 \mathcal{B} &= 
 \left(\begin{array}{cc}
      -\Delta^{-1} &  \\
       & I^{-1} + (-\nabla\cdot(\alpha \DW^2\nabla))^{-1} 
 \end{array}\right),
\\
 \mathcal{B}_H &= 
 \left(\begin{array}{cc}
      -\Delta^{-1} &  \\
       & I^{-1} + (-\nabla_H\cdot(\alpha \DW^2\nabla_H))^{-1} 
 \end{array}\right). \label{eq:precond_varW_coarse}
\end{align}
In \Cref{ex:varW} the preconditioner in \eqref{eq:precond_varW} is compared to \eqref{eq:precond_espen} where constant $W$ weighting is used.

\begin{example}[Variable $W$]\label{ex:varW}
  We consider Stokes flow \eqref{eq:stokes} in a channel domain with constrictions, as shown in \Cref{fig:varW}.
  The constrictions are parameterized in terms of the parameter $0.1 \leq r \leq 0.49$ such that $r=0.5$
  represents a blocked channel.

  Given the domain, we let $\partial\Omega_D$ be the top and bottom boundaries and compare the preconditioner $\mathcal{B}$
  in \eqref{eq:precond_espen} which uses the constant scaling $W=\min_{\mathbf{x}}W(\mathbf{x})$ with preconditioner \eqref{eq:precond_varW}. As in the previous example we pick $\alpha=1$ for both preconditioners.
  Interestingly, \Cref{tab:varW} reveals that in this example constant scaling leads
  to faster convergence for $0.1 \leq r \leq 0.4$ while for very tight constrictions
  $r\geq 0.45$, only the variable-thickness preconditioner yields domain-robustness. Specifically,
  we observe that with the latter, the iterations remain bounded in mesh and $r$ at roughly $100$ while
  with constant scaling, the iterations grow rapidly as the constrictions tighten. We remark that the geometries
  are discretized such that the element size $h$ (also at the coarsest refinement level, $l=0$) is smaller than $\min_{\mathbf{x}}W(\mathbf{x})$.
  
\begin{figure}
  \centering
  \includegraphics[width=0.25\textwidth]{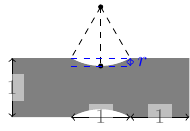}
  \includegraphics[width=0.9\textwidth]{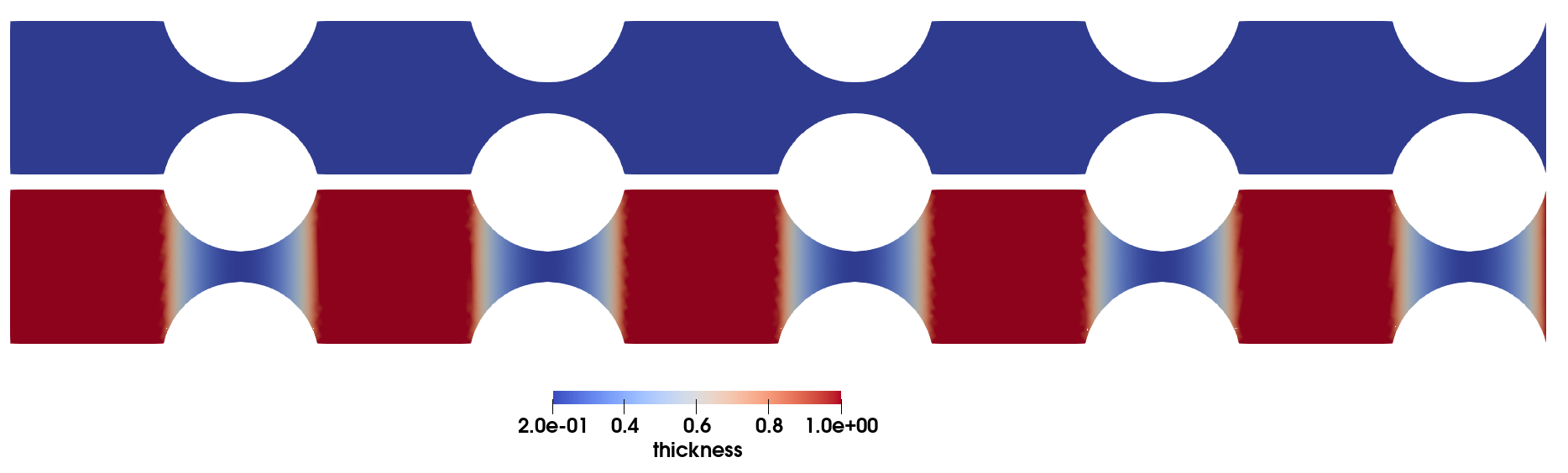}
  \caption{
    Channel with constrictions. (Top) Parameterization of the constriction in terms of parameter $r$.
    (Bottom) Thickness functions $W$ for case $r=0.4$. With \eqref{eq:precond_espen} minimal thickness
    is used, $\min_{\mathbf{x}}W(\mathbf{x})$, while with \eqref{eq:precond_varW} $W$ varies in space.
  }
  \label{fig:varW}
\end{figure}

\begin{table}
  \centering
  \caption{
    Performance of constant thickness preconditioner \eqref{eq:precond_espen} and variable-thickness  preconditioner \eqref{eq:precond_varW}
    for Stokes flow in channel with constrictions \Cref{fig:varW}.
    In \eqref{eq:precond_espen} we set $\min_{\mathbf{x}}W(\mathbf{x})$.
    Discretization by $[\mathbb{P}_2]^2\times\mathbb{P}_1$.    
  }
  \footnotesize{
  \begin{tabular}{c|lll||lll}
    \hline
    & \multicolumn{3}{c||}{Constant $W$ \eqref{eq:precond_espen}} & \multicolumn{3}{c}{Variable $W$ \eqref{eq:precond_varW}} \\
    \hline
    \backslashbox{$r$}{$l$} & 1 & 2 & 3 & 1 & 2 & 3\\
    \hline
0.1 & 15.4(66) & 16.52(69) & 17.24(73)           & 19.5(69) & 20.7(74) & 21.43(77)\\     
0.2 & 13.37(63) & 15.74(72) & 16.36(77)          & 22.11(72) & 23.43(81) & 23.72(87)\\   
0.3 & 15.88(70) & 16.62(74) & 16.88(76)          & 26.59(83) & 25.93(91) & 25.74(95)\\   
0.4 & 20.29(82) & 20.09(85) & 19.93(84)          & 27.88(92) & 26.52(97) & 25.84(100)\\  
0.45 & 32.9(130) & 32.06(128) & 31.48(120)       & 25.86(98) & 25.25(100) & 24.93(97)\\  
0.46 & 39.63(142) & 38.8(160) & 38.24(142)       & 25.53(99) & 24.92(102) & 24.68(100)\\ 
0.47 & 51.55(187) & 50.4(194) & 49.82(176)       & 25.21(102) & 24.61(100) & 24.43(97)\\ 
0.48 & 75.49(270) & 74.33(285) & 73.7(258)       & 24.6(102) & 24.26(98) & 24.17(96)\\   
0.49 & 55.82(551) & 55.3(546) & 148(521)         & 24.11(104) & 23.97(104) & 23.92(102)\\
\hline
  \end{tabular}}
  \label{tab:varW}
\end{table}
  
\end{example}


In \Cref{sec:coarse}, we have derived, through dimensional analysis, a reduced Stokes flow model, which simplifies to the Laplace equation in channel direction with a specific coefficient, namely, $\tfrac{1}{12}W^2$.
Here, the factor $\tfrac{1}{12}$ arises from the channel geometry (``parallel plates'') and the boundary conditions.
However, in the examples thus far, this prefactor $\alpha$ has been ignored. Therefore, we next investigate preconditioners  \eqref{eq:precond_espen} and \eqref{eq:precond_manfred} with varying $\alpha$.

\begin{example}[Optimal scaling]\label{ex:optimize}
  We consider Stokes flow \eqref{eq:stokes} in channel domains $\Omega_{W}=(0, L)\times(0, W)$, $L=10$ and $W=1, 0.25, 0.125$. 
  Here, the channel height is varied so that several domains are investigated for which the model order reduction assumption in \Cref{sec:coarse}, 
  i.e. $\epsilon=W/L\ll 1$, are valid to different extents. As in \Cref{ex:noslip} we let $\partial\Omega_N$ be the $W$-long lateral boundaries and
  prescribe no-slip conditions on top and bottom boundaries.

  We discretize the problem by $[\mathbb{P}_2]^2\times\mathbb{P}_1$ and $[\mathbb{P}_2]^2\times\mathbb{P}_0$ elements (FE), and staggered-grid finite volumes (FV) \cite{Harlow1965} (to investigate if
  discretization plays a role). For the FV scheme, we also compare preconditioners \eqref{eq:precond_espen} and \eqref{eq:precond_manfred}. Applying the solver settings from \Cref{ex:noslip},
  we plot in \Cref{fig:optimize} (FE) and \Cref{fig:optimize-fv} (FV) the dependence of the condition number and MinRes iterations on $0 < \alpha \leq 4$.
  For the FE schemes the finest meshes, $l=4$ in \Cref{ex:noslip} corresponding to $h=3\cdot 10^{-2}$, are used. For the FV scheme, uniform Cartesian grids with $h=0.02$ are used.
  We observe that with all discretization schemes, the results are similar for $W=0.25$ and $W=0.125$ while the dependence on $\alpha$ is somewhat
  different for $W=1$. Focusing on $W=0.125$, we observe that the MinRes vs $\alpha$ dependence has a pronounced minimum:
  $54$ iterations for $\tfrac{1}{6} < \alpha < \tfrac{1}{5}$ (cf. $115$ iterations for $\alpha=1$) using $[\mathbb{P}_2]^2\times\mathbb{P}_1$, 
  $52$ iterations for $\tfrac{1}{4} < \alpha < \tfrac{1}{2}$ (cf. $87$ iterations for $\alpha=1$) using $[\mathbb{P}_2]^2\times\mathbb{P}_0$, and $33$ iterations for $\tfrac{1}{12} < \alpha < \tfrac{1}{8}$ using finite volumes (cf. $100$ iterations for $\alpha = 1$). In the condition
  number dependence, the curves are rather flat for $\alpha$ smaller than the minimizer value $\alpha^{*}$: condition number $7.4$ for $\alpha^{*}=0.25$
  (value $8.3$ is attained for $\alpha=\tfrac{1}{12}$ and $\alpha=1$ yields $23.5$) using $[\mathbb{P}_2]^2\times\mathbb{P}_1$ while with $[\mathbb{P}_2]^2\times\mathbb{P}_0$
  the minimal condition number of $7.4$ is obtained at $\alpha^*=\tfrac{1}{3}$ whereas $\alpha=\tfrac{1}{12}$ yields $9.4$ and $\alpha=1$ yields $14.0$. The minimal condition number of ca. $5.0$ is obtained for $\tfrac{1}{12} < \alpha < \tfrac{1}{8}$ using finite volumes (both for $\mathcal{B}$ and $\mathcal{B}_H$). Similar values of $\alpha$ are optimal for both $\mathcal{B}$ and $\mathcal{B}_H$.

  In summary, we observe that optimizing $\alpha$ in \eqref{eq:precond_espen} leads to significantly improved performance over the choice $\alpha = 1$.
  Moreover, the $\alpha$ derived with model order reduction techniques in \Cref{sec:coarse} appears close to optimal (and even optimal for the finite volume discretization).

  \begin{figure}[!htb]
    \centering
    \includegraphics[width=0.45\textwidth]{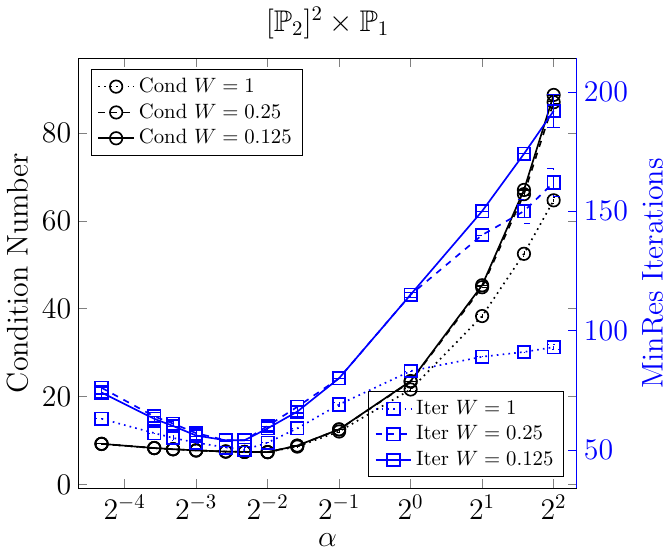}
    \hfill
    \includegraphics[width=0.45\textwidth]{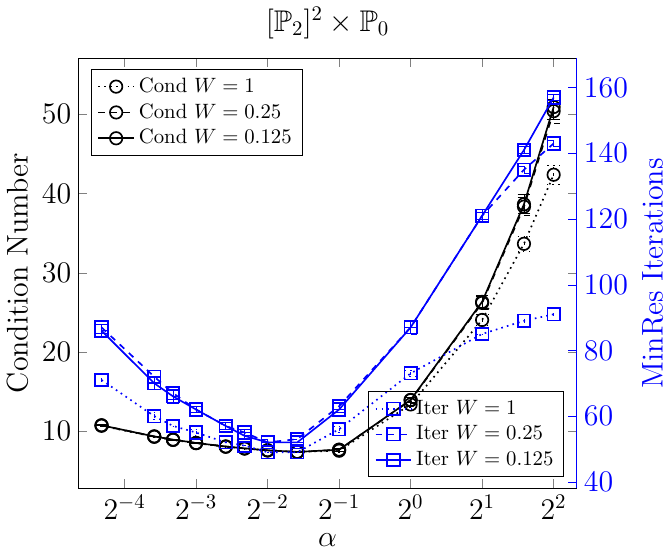}    
    \vspace{-10pt}
    \caption{
      Performance of preconditioners \eqref{eq:precond_espen} for 
      Stokes problem in the channel domain $(0, 10)\times(0, W)$ with no-slip boundary conditions on top and bottom.
      Discretization by (left) $[\mathbb{P}_2]^2\times\mathbb{P}_1$ and (right) $[\mathbb{P}_2]^2\times\mathbb{P}_0$ elements.
      In the subplots, condition numbers and MinRes iterations are plotted on the left and right vertical axes, respectively.
      The values are obtained on meshes with refinement level $l=4$ ($h=3\cdot10^{-2}$) with the errorbars representing
      difference to level $l=3$ results.
    }
    \label{fig:optimize}
  \end{figure}

  \begin{figure}[!htb]
    \centering
    \includegraphics[width=0.45\textwidth]{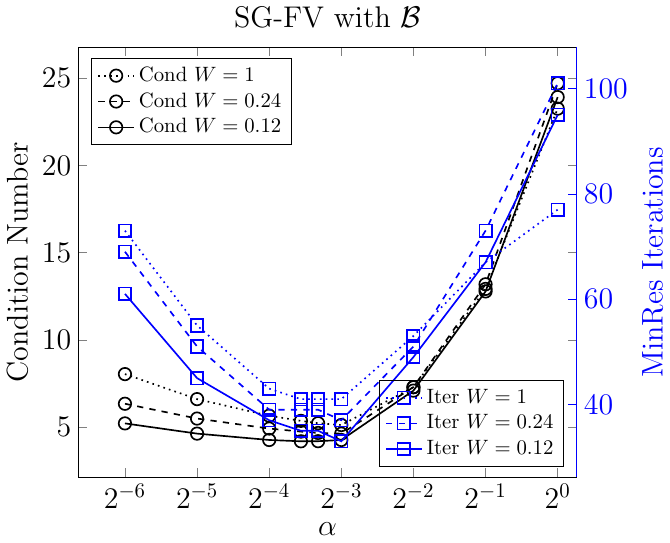}
    \hfill
    \includegraphics[width=0.45\textwidth]{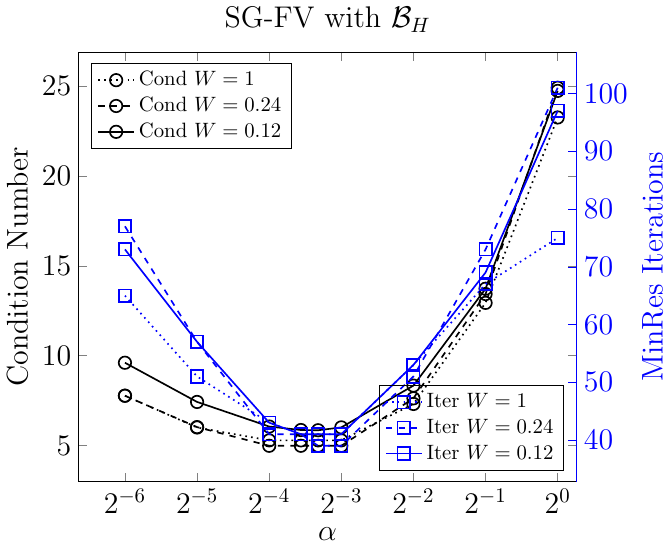}   
    \vspace{-10pt}
    \caption{
      Performance of preconditioners \eqref{eq:precond_espen} and \eqref{eq:precond_manfred} for 
      Stokes problem in the channel domain $(0, 10)\times(0, W)$ with no-slip boundary conditions on top and bottom. Discretization by staggered-grid finite volumes. In the subplots, condition numbers and MinRes iterations are plotted on the left and right vertical axes, respectively.
      The results are obtained on regular Cartesian grids with $h=0.02$.
    }
    \label{fig:optimize-fv}
  \end{figure} 
\end{example}

The derived reduced Stokes flow model also suggests that the $\tfrac{1}{12}W^2$ scaling in $2d$ channel domains is relevant only in the long channel direction. Hence, it is insightful to investigate preconditioners of the form
\begin{equation}\label{eq:precond_aniso}
 \mathcal{B} = 
 \left(\begin{array}{cc}
      -\Delta^{-1} &  \\
       & I^{-1} + \left(-\nabla\cdot\left(\mbox{\scriptsize$\begin{bmatrix}\alpha_L&\\&\alpha_W\end{bmatrix}$}\DW^2\nabla\right)\right)^{-1} 
 \end{array}\right).
\end{equation}

\begin{example}[Anisotropic scaling]\label{ex:anistropy}
We consider Stokes flow \eqref{eq:stokes} in channel domains  $\Omega_W=(0, L)\times(0, W)$, $L=10$ and $W=1, 0.24, 0.12$. We use a staggered-grid finite-volume discretization~\cite{Harlow1965}. We fix $\alpha_L = \tfrac{1}{12}$ and set $\alpha_W = \tfrac{\beta}{12} (\tfrac{L}{W})^2$, with $\beta$ varied over several orders of magnitudes. We compute on a Cartesian grid with square elements of edge length $h = 0.02$.

The performance results for preconditioner \eqref{eq:precond_aniso} are shown in \cref{tab:aniso}.
\begin{table}[!htb]
    \centering
    \footnotesize
    \caption{Performance of anisotropic preconditioner \eqref{eq:precond_aniso} for Stokes flow in a straight channel. Discretization by staggered-grid finite volumes. Condition numbers and MinRes iterations (in parentheses).}
    \footnotesize{
    \begin{tabular}{c|c|c|c|c|c|c}
         \hline
\diagbox{$W$}{$\beta$} & 1000 & 100 & 10 & 1 & 0.1 & 0.01 \\ \hline
                     1 & $8.9$ ($59$) & $5.1$ ($41$) & $5.2$ ($41$) & $5.3$ ($41$) & $5.3$ ($41$) & $5.3$ ($41$) \\
                  0.24 & $26.7$ ($187$) & $9.8$ ($85$) & $4.9$ ($39$) & $4.7$ ($39$) & $4.9$ ($39$) & $5.0$ ($39$) \\
                  0.12 & $45.7$ ($342$) & $15.7$ ($161$) & $6.1$ ($53$) & $4.4$ ($33$) & $4.5$ ($33$) & $4.5$ ($33$)\\ \hline
    \end{tabular}}
    \label{tab:aniso}
\end{table}
We observe that the preconditioner is robust with respect to aspect ratio for $\beta \leq 1$.
A further decrease in $\beta$ (corresponding to neglecting derivatives in the short direction for the Laplace operator) does not impact performance.
This, in turn, means that since $W \ll L$, an isotropic Laplacian with scaling $\tfrac{1}{12}W^2$ (as previously shown) is robust. However, we can demonstrate
here that it does not seem necessary in practice to add a contribution to the mass matrix pressure preconditioner regarding the short direction. This is also reflected in the fact that preconditioner $\mathcal{B}_H$ in \cref{ex:channels} is robust in the domain aspect ratio, although its coarse grid pressure operator only contains contributions from derivatives in channel direction.
\end{example}


We finally apply the domain-robust preconditioners \eqref{eq:precond_espen} and \eqref{eq:precond_varW} to Stokes flow in $3d$ domains.
Unlike with previous $2d$ examples, the $3d$ setting will sometimes preclude the use of exact (LU-based)
preconditioners. To this end, we will also test approximating the presented preconditioners using algebraic multigrid (AMG).
\begin{example}[Tesla valve]\label{ex:3d}
  We consider the Stokes flow in a Tesla valve, a passive device/geometry that shows preferential directionality
  of the flow when Navier-Stokes equations are assumed, see e.g. \cite{nguyen2021early} and references therein.
  Here, the flow shall be driven by pressure drop as we set $\ffthree=-2\nn$, $\ffthree=-1\nn$ respectively on the inlet
  and outlet boundaries $\partial\Omega_N$. All the remaining boundaries have $\uu=\mathbf{0}$ prescribed and we set $\ff=\mathbf{0}$ in \eqref{eq:stokes1}.
  We let $W=0.2$ be the width/height of the valve's square cross-section. Further, the problem geometry shall be parameterized in terms of
  the number of segments/diodes ($S$). Concerning the Stokes inf-sup condition, we note that the distance between the inlet
  and outlet grows with the number of segments. A Tesla valve with $S=2$ segments is shown in \Cref{fig:tesla}. Here, the shortest
  distance between the inlet and outlet is approximately $10$.

  Following discretization by the lowest order Taylor-Hood ($[\mathbb{P}_2]^2\times\mathbb{P}_1$) elements we solve \eqref{eq:stokes}
  with preconditioned MinRes solver using $10^{-10}$ as the convergence threshold for the preconditioned residual norm. Depending on the
  refinement level and the number of segments, the system size here ranges between $8\cdot 10^3$ and $3.5\cdot10^6$ unknowns.
  In \Cref{tab:tesla} we report the number of iterations required for convergence with the preconditioner \eqref{eq:precond_espen} with $\alpha=1$ and $\alpha=\tfrac{1}{12}$, cf. \Cref{sec:coarse}. In both cases, both its exact
  realization is considered as well as an inexact one where the inverses are approximated in terms of one V-cycle of BoomerAMG\footnote{We apply the same (default)
  AMG settings to all discrete operators, i.e. $-\Delta^{-1}$ over $\VV_h$, and $I^{-1}$ and $-\Delta^{-1}$ over $Q_h$.} solver \cite{boomeramg}.
  We observe that all the preconditioners yield iteration counts that appear to be bounded in the number of segments. With the ``optimized''
  $\alpha=\tfrac{1}{12}$ the convergence is then clearly faster than with $\alpha=1$. The difference is
  especially noticeable when the AMG realizations are compared. Interestingly, the exact and inexact versions of the case $\alpha=\tfrac{1}{12}$
  perform similarly.

  We remark that with the standard Stokes preconditioner \eqref{eq:B_standard}, between $371$ and $471$ MinRes iterations are required with $S=1$ segment,
  while with $S=2$ the iterations exceed $600$. The corresponding estimated\footnote{Here the condition number estimates are obtained
  during the MinRes iterations from the Lanczos process instead of through separate computations of extremal eigenvalues as in the earlier
  examples.}
  condition numbers are then $8\cdot 10^3$ and $24\cdot 10^3$ respectively.

\begin{figure}
  \centering
    \includegraphics[width=1\textwidth]{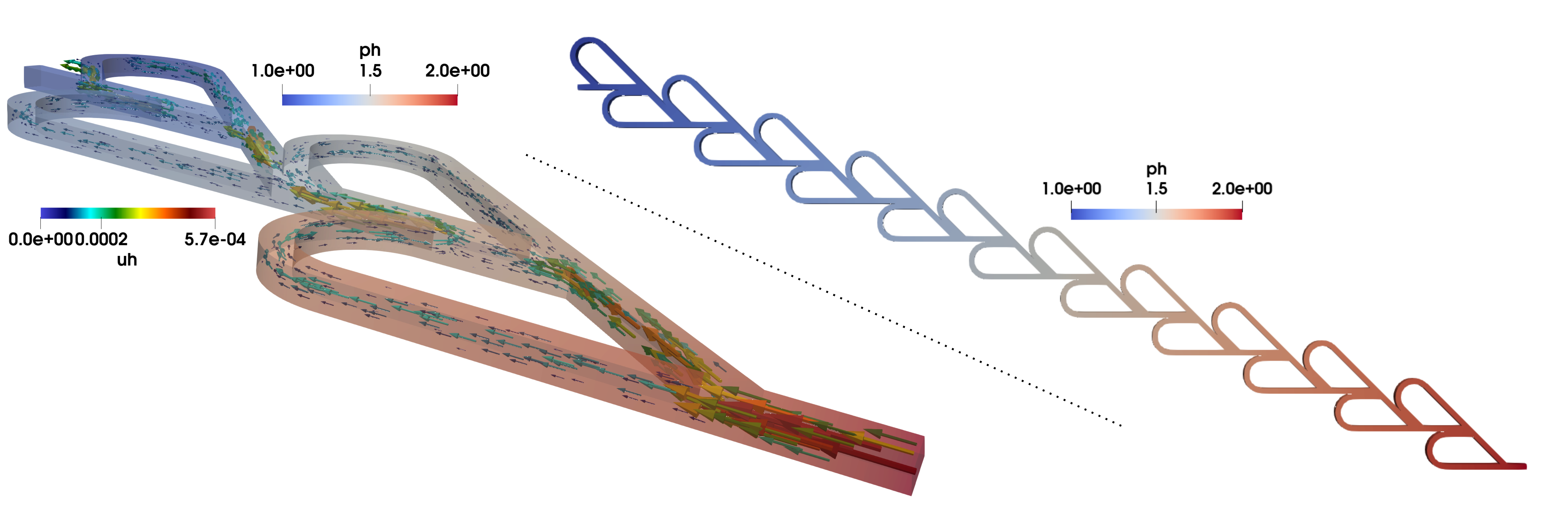}
    \vspace{-10pt}
    \caption{Pressure-driven Stokes flow in a $3d$ Tesla-valve-inspired domain with $2$ segments (left) and $10$ segments (right).}
    \label{fig:tesla}
  \end{figure}

\begin{table}
  \centering
  \caption{
    Number of MinRes iterations required for Stokes flow in Tesla valves (cf. \Cref{fig:tesla}) with different number of segments $S$.
    Exact (LU-based) and inexact (AMG-based) realizations of preconditioners \eqref{eq:precond_espen} and \eqref{eq:precond_varW} are
    compared. Shown in parentheses is the condition number estimated on the MinRes-built Krylov subspace. Refinement level is denoted by $l$,
    while -- indicates simulations that ran out of memory.
  }
  \footnotesize{
        \renewcommand{\arraystretch}{1.1}
        \setlength{\tabcolsep}{2pt}    
  \begin{tabular}{c|cccc||cccc}
    \hline
    & \multicolumn{4}{c||}{\eqref{eq:precond_espen} with $\alpha=1$ by LU} & \multicolumn{4}{c}{\eqref{eq:precond_espen} with $\alpha=1/12$ by LU} \\
    \hline
    \backslashbox{$S$}{$l$} & 1 & 2 & 3 & 4 & 1 & 2 & 3 & 4\\
    \hline
1 & 113(61.17) & 125(53.42) & 139(52.72) & 143(53.25) & 87(39.83) & 91(33.39) & 95(28.31) & 91(26.53)\\     
2 & 119(63.37) & 125(53.35) & 139(53.05) & 143(53.5)  & 87(40.38) & 89(31.18) & 95(29.38) & 91(25.83)\\     
3 & 119(63.41) & 125(53.56) & 139(53.27) & 143(53.68) & 87(39.93) & 89(31.75) & 95(28.49) & 91(26.38)\\     
4 & 117(64.01) & 125(53.81) & 139(53.21) & 143(53.71) & 89(40.69) & 91(37.53) & 95(28.68) & 91(26.42)\\     
8 & 119(63.03) & 125(53.87) & 139(53.45) & --        & 87(39.61) & 93(42.7) & 95(29.57) & --\\         
\hline
\hline
    & \multicolumn{4}{c||}{\eqref{eq:precond_espen} with $\alpha=1$ by AMG} & \multicolumn{4}{c}{\eqref{eq:precond_espen} with $\alpha=1/12$ by AMG} \\
    \hline
    \backslashbox{$S$}{$l$} & 1 & 2 & 3 & 4 & 1 & 2 & 3 & 4\\
    \hline
2 & 127(80.72) & 156(83.79) & 178(82.66) & 192(91.34) &   91(41.25) & 100(34.08) & 101(27.38) & 99(25.17)\\ 
4 & 131(82.91) & 164(91.92) & 185(99.28) & 198(95.16) &   93(43.43) & 98(32.06) & 103(28.27) & 100(24.55)\\ 
6 & 132(84.41) & 166(89.66) & 188(91.61) & 206(97.09) &   92(37.11) & 100(33.29) & 105(28.46) & 102(25.71)\\
8 & 140(88.43) & 162(86.98) & 198(100) & 210(107)     &   97(48.13) & 111(38.16) & 106(28.37) & 102(26.03)\\
16 & 141(89.67) & 203(91.7) & 208(120) & 216(111)      &  99(48.61) & 131(44.72) & 106(28.73) & 101(24.46)\\
\hline
  \end{tabular}}
  \label{tab:tesla}
\end{table}
\end{example}



\begin{example}[Symmetric gradient]\label{ex:symgrad}
Consider the Stokes problem
\begin{equation}
\begin{aligned}
  -\nabla\cdot(\nabla^S\uu) - \nabla p &= \ffone&&\mbox{ in } \Omega\subset\mathbb{R}^d,\\
  \nabla \cdot \uu &= 0, &&\mbox{ in } \Omega,  \\
 \uu &= \fftwo, &&\mbox{ on } \partial \Omega_D, \\ 
 (\nabla^{S}\uu-pI)\cdot\nn &= \ffthree, &&\mbox{ on } \partial \Omega_N.
\end{aligned}
\label{eq:sym_stokes}
\end{equation}
  Mirroring the setup of \Cref{ex:noslip} the system \eqref{eq:sym_stokes} shall be solved on domain $(0, L)\times (0, 1)$ with $\partial\Omega_N$ the left and right edges.
  In addition to preconditioner \eqref{eq:precond_espen} we also investigate
  its modification given by 
  \begin{equation}\label{eq:precond_espen_sym}
 \mathcal{B} = 
 \left(\begin{array}{cc}
      (-\nabla\cdot\nabla^S)^{-1} &  \\
       & I^{-1} + (-\alpha\DW^2\Delta)^{-1} 
 \end{array} \right),
  \end{equation}
  where the velocity norm is induced by the symmetric gradient operator. We remark
  that with both preconditioners $\alpha=1$ in the numerical experiments.

  Discretizing the system by Taylor-Hood elements it can be seen in \Cref{tab:symgrad}
  that the two preconditioners yield similar condition numbers which, crucially, appear
  to be bounded in $L$. We also observe that the symmetric-gradient variant 
  leads to faster MinRes convergence with \eqref{eq:precond_espen_sym} requiring about half
  the iterations compared to \eqref{eq:precond_espen}. This observations holds already for
  $L=1$. We remark that about 350 MinRes iterations are required if the standard preconditioner
  \eqref{eq:B_standard} is used.

\begin{table}
  \centering
  \caption{Performance of preconditioners \eqref{eq:precond_espen} and \eqref{eq:precond_espen_sym}
    for Stokes problem \eqref{eq:sym_stokes} with symmetric gradient on domain $(0, L)\times (0, 1)$ and mixed boundary conditions, where traction is set on the short left and right edges and
  no-slip condition \eqref{eq:freeslip} is prescribed on long boundaries. Condition number of the preconditioned problem
  together with MinRes iterations (in parentheses)
  is shown for different levels of mesh refinement $l$. Discretization by Taylor-Hood $[\mathbb{P}_2]^2\times\mathbb{P}_1$ elements. Solver
  setup is identical to \Cref{ex:noslip}.
  }
\label{tab:symgrad}
\footnotesize
\setlength{\tabcolsep}{3pt}    {
  \begin{tabular}{c|llll||llll}
    \hline
    & \multicolumn{4}{c||}{Stokes \eqref{eq:sym_stokes} with \eqref{eq:precond_espen}} & \multicolumn{4}{c}{Stokes \eqref{eq:sym_stokes} with \eqref{eq:precond_espen_sym}}\\
    \hline
    \diagbox{$L$}{$l$} & 1 & 2 & 3 & 4 & 1 & 2 & 3 & 4\\
    \hline
1 & 5.59(85) & 5.6(85) & 5.64(86) & 5.63(86)          &  5.99(45) & 5.99(46) & 6.01(47) & 6.01(47)\\    
3 & 7.62(81) & 8.05(85) & 8.28(93) & 8.4(95)          &  7.81(45) & 7.97(46) & 8.09(51) & 8.21(55)\\    
5 & 9.76(88) & 10.25(93) & 10.51(97) & 10.66(100)     &  9.54(49) & 10.13(50) & 10.4(56) & 10.55(59)\\  
10 & 11.87(103) & 12.24(108) & 12.44(115) & 12.55(115)&  11.78(56) & 12.26(63) & 12.48(69) & 12.58(69)\\
20 & 12.86(120) & 13.09(124) & 13.21(131) & 13.26(128)&  12.84(66) & 13.19(70) & 13.32(79) & 13.37(75)\\
50 & 13.29(125) & 13.45(129) & 13.5(138) & 13.52(137) &  13.31(69) & 13.58(77) & 13.65(84) & 13.67(84)\\
\hline
  \end{tabular}}
\end{table}
\end{example}

\begin{example}[Non-unit viscosity] We return to the Laplacian-based formulation of
the Stokes problem where, newly, we consider constant viscosity $\mu\neq 1$ and assess domain-
as well as parameter-robustness of the preconditioner. To this end we study a
pair of operators
\begin{equation}\label{eq:mu_var}
 \mathcal{A} = 
 \left(\begin{array}{cc}
      -\mu\Delta &  \nabla \\
       \nabla\cdot & 
 \end{array} \right), \quad
 \mathcal{B} = 
 \left(\begin{array}{cc}
      (-\mu\Delta)^{-1} &  \\
       & (\tfrac{1}{\mu}I)^{-1} + (-\alpha\tfrac{\DW^2}{\mu}\Delta)^{-1} 
 \end{array} \right).
 \end{equation}
With the domain and boundary condition setup as in \Cref{ex:noslip} we report in \Cref{tab:mu_var}
the condition numbers and MinRes iteration counts for \eqref{eq:mu_var} and viscosity values
$\mu=10^{-4}, 10^{-8}$. Comparing the results with $\mu=1$ in \Cref{tab:espen_manfred_TD} the
condition numbers in the three cases are almost identical. Convergence of MinRes with $\mu=10^{-4}, 10^{-8}$
is somewhat slower than when $\mu=1$, however, the iterations appear bounded in $\mu$ and $L$.
\begin{table}
  \centering
  \caption{
    Performance preconditioners \eqref{eq:precond_espen} for Stokes problem with varying viscosity
    \eqref{eq:mu_var} on domain $(0, L)\times (0, 1)$ using mixed boundary conditions, where traction is set on the short left and right edges and
    no-slip condition \eqref{eq:freeslip} is prescribed on long boundaries. Condition number of the preconditioned problem
    together with MinRes iterations (in parentheses)
  is shown for different levels of mesh refinement $l$. Discretization by Taylor-Hood $[\mathbb{P}_2]^2\times\mathbb{P}_1$ elements. Solver
  setup is identical to \Cref{ex:noslip}.
  }
\label{tab:mu_var}
\footnotesize
\setlength{\tabcolsep}{2.5pt}    {
  \begin{tabular}{c|llll||llll}
    \hline
    & \multicolumn{4}{c||}{$\mu = 10^{-4}$} & \multicolumn{4}{c}{$\mu = 10^{-8}$}\\
    \hline
    \diagbox{$L$}{$l$} & 1 & 2 & 3 & 4 & 1 & 2 & 3 & 4\\
    \hline
1 & 5.51(47) & 5.5(47) & 5.87(47) & 5.58(47)            & 5.51(47) & 5.5(47) & 5.87(47) & 5.58(47)\\         
3 & 11.77(51) & 12.54(51) & 12.95(55) & 13.18(56)       & 11.77(51) & 12.54(51) & 12.95(55) & 13.18(58)\\    
5 & 16.06(57) & 17(57) & 17.51(61) & 17.79(63)          & 16.06(57) & 17(57) & 17.51(61) & 17.79(64)\\       
10 & 20.28(73) & 21(74) & 21.4(84) & 21.61(81)          & 20.28(69) & 21(71) & 21.4(78) & 21.61(77)\\        
20 & 22.23(90) & 22.66(89) & 22.87(97) & 22.99(91)      & 22.23(88) & 22.66(80) & 22.87(96) & 22.99(89)\\    
50 & 22.17(99) & 22.52(107) & 22.65(109) & 22.69(112)   & 22.26(100) & 22.52(107) & 22.64(109) & 22.77(110)\\
100 & 22.2(99) & 22.55(104) & 22.63(110) & 22.66(110)   & 22.2(99) & 22.55(104) & 22.71(108) & 22.69(112)\\
\hline
  \end{tabular}}
\end{table}
\end{example}

\section{Acknowledgements}
We wish to thank the two referees for their careful reading of the manuscript and
their valuable comments that helped improve the paper.
This work was funded by the European Research Council under grant 101141807 (aCleanBrain).
T.K. additionally acknowledges funding from EU’s Horizon 2020 Research and Innovation programme under the Marie Skłodowska-Curie Actions Grant 801133.
K.A.M additionally acknowledges funding by: 
 Stiftelsen Kristian Gerhard Jebsen via the K. G. Jebsen Centre
 for Brain Fluid Research; the national infrastructure for computational science in Norway Sigma2 via grant NN9279K; the Center of Advanced Study at the Norwegian Academy of Science and Letters under the program Mathematical Challenges in Brain Mechanics,  research stay at  ICERM, Brown under the program 
Numerical PDEs: Analysis, Algorithms, and Data Challenges, and the “Computational Hydrology project” a strategic Sustainability initiative at the Faculty of Natural Sciences, UiO.

\appendix
\section{Convergence of applied discretizations}\label{sec:cvrg}
To verify the implementation of the Stokes solvers used throughout the manuscript,
we consider \eqref{eq:stokes} with $\Omega=(0, 2)\times(0, 1)$, with $\partial\Omega_D$
the top and bottom edges. The source term $\ff$ and boundary data $\uu_D$ and $\ffthree$
are computed based on the solution
\begin{equation}\label{eq:mms}
  p = \sin(2\pi(x-y)),\quad \uu=\left(\frac{\partial \phi}{\partial  y}, -\frac{\partial \phi}{\partial  x}\right)^T\quad\text{where}\quad\phi = \cos(\pi(2x-y)).
\end{equation}
The problem is discretized with conforming elements $[\mathbb{P}_2]^2\times\mathbb{P}_1$, $[\mathbb{P}_2]^2\times\mathbb{P}_0$
and non-conforming $[\mathbb{CR}_1]^2\times\mathbb{P}_0$ pair employing the facet stabilization \cite{burman2005stabilized}, where the implementation is based on FEniCS~\cite{logg2012automated}, or a staggered grid finite volume scheme~\cite{Harlow1965}, where the implementation is based on DUNE/DuMu\textsuperscript{x}~\cite{Dune210,Koch2021}. Velocity and pressure errors obtained on a series of uniformly
refined unstructured meshes are reported in \Cref{tab:cvrg}.

\renewcommand{\arraystretch}{1.2}
\begin{table}
  \centering
  \caption{
    Error convergence of different spatial discretization schemes used in this paper for the Stokes problem \eqref{eq:stokes}
    setup in \Cref{sec:cvrg}. The estimated order of convergence based on the current and previous refinement levels is shown in parentheses. The three finite element schemes use unstructured grids; the staggered-grid finite-volume scheme (FV-SG) uses a structured Cartesian grid. The $L^2$-norm for FV-SG is computed by numerical integration using a single quadrature node at the cell center. In this norm, second-order (super-)convergence is expected.
  }
  \footnotesize{
    \begin{tabular}{c||ccc}
      \hline
      & $[\mathbb{P}_2]^2\times\mathbb{P}_1$ & $[\mathbb{CR}_1]^2\times\mathbb{P}_0$ & $[\mathbb{P}_2]^2\times\mathbb{P}_0$\\
      \hline\hline
      $h$ & \multicolumn{3}{c}{$\lVert \uu - \uu_h \rVert_{\HH^1}$}\\
      \hline
1.85E-01 & 2.386E+00(--)   & 1.343E+01(--)    & 2.392E+00(--)   \\
1.08E-01 & 9.678E-01(1.67) & 8.616E+00(0.82)  & 9.746E-01(1.67) \\
5.10E-02 & 2.523E-01(1.79) & 4.408E+00(0.89)  & 2.585E-01(1.77) \\
2.68E-02 & 6.466E-02(2.12) & 2.211E+00(1.07)  & 7.052E-02(2.02) \\
1.35E-02 & 1.571E-02(2.06) & 1.105E+00(1.01)  & 2.113E-02(1.76) \\
6.76E-03 & 4.035E-03(1.96) & 5.524E-01(1.00)  & 8.141E-03(1.37) \\
3.30E-03 & 1.009E-03(1.93) & 2.763E-01(0.97)  & 3.679E-03(1.11) \\
      \hline
      $h$ & \multicolumn{3}{c}{$\lVert p - p_h \rVert_{L^2}$}\\
      \hline
1.85E-01 & 1.636E-01(--)   & 1.488E+00(--)   & 2.397E-01(--)    \\
1.08E-01 & 4.294E-02(2.48) & 7.344E-01(1.31) & 1.253E-01(1.20)\\
5.10E-02 & 7.742E-03(2.29) & 2.228E-01(1.59) & 5.742E-02(1.04)\\
2.68E-02 & 1.369E-03(2.70) & 6.717E-02(1.87) & 2.834E-02(1.10)\\
1.35E-02 & 2.383E-04(2.55) & 2.202E-02(1.63) & 1.416E-02(1.01)\\
6.76E-03 & 4.308E-05(2.47) & 8.598E-03(1.36) & 7.080E-03(1.00)\\
3.30E-03 & 9.749E-06(2.07) & 3.880E-03(1.11) & 3.540E-03(0.97)\\
\hline\hline
   


    & \multicolumn{3}{c}{FV-SG} \\
    \hline\hline
    $h$ & $\lVert p - p_h \rVert_{L^2}$ & $\lVert u_x - u_{x,h} \rVert_{L^2}$ & $\lVert u_y - u_{y,h} \rVert_{L^2}$ \\ \hline
    2.00e-01 & 1.84e+00(-) & 2.00e-01(-) & 2.74e-01(-)\\
1.00e-01 & 5.54e-01(1.73) & 4.17e-02(2.26) & 6.12e-02(2.16)\\
5.00e-02 & 1.50e-01(1.88) & 9.92e-03(2.07) & 1.44e-02(2.08)\\
2.50e-02 & 3.85e-02(1.96) & 2.45e-03(2.02) & 3.52e-03(2.04)\\
1.25e-02 & 9.70e-03(1.99) & 6.10e-04(2.00) & 8.71e-04(2.01)\\
6.25e-03 & 2.43e-03(2.00) & 1.52e-04(2.00) & 2.17e-04(2.01)\\
3.13e-03 & 6.08e-04(2.00) & 3.81e-05(2.00) & 5.41e-05(2.00)\\
\hline
  \end{tabular}
  \label{tab:cvrg}
}
\end{table}

\bibliographystyle{siamplain}
\bibliography{references}

\end{document}